\documentclass[11pt, letterpaper]{article}
\usepackage{fullpage}
\usepackage{amsmath,amssymb,amsthm,enumerate,bm,bbm,color,setspace}

\newtheorem{thm}{Theorem}[section]
\newtheorem{lem}[thm]{Lemma}
\newtheorem{cor}[thm]{Corollary}
\newtheorem{prop}[thm]{Proposition}

\theoremstyle{remark}

\numberwithin{equation}{section}

\def\ind{{\bf 1}}

\newcommand{\symdif}{\vartriangle}

\numberwithin{equation}{section}

\newenvironment{definition}[1][Definition]{\begin{trivlist}
\item[\hskip \labelsep {\bfseries #1}]}{\end{trivlist}}

\newenvironment{remark}[1][Remark:]{\begin{trivlist}
\item[\hskip \labelsep {\bfseries #1}]}{\end{trivlist}}
\newenvironment{question}[1][Question:]{\begin{trivlist}
\item[\hskip \labelsep {\bfseries #1}]}{\end{trivlist}}

\newcommand{\A}{\mathbb{A}}
\newcommand{\Q}{\mathbb{Q}}
\newcommand{\R}{{\mathbb{R}}}
\newcommand{\T}{{\mathbb{T}}}

\newcommand{\Z}{{\mathbb{Z}}}
\newcommand{\N}{{\mathbb{N}}}

\newcommand{\beq}{\begin{equation}}
\newcommand{\eeq}{\end{equation}}

\def\E{{\mathbb E}}
\newcommand{\supp}{\textnormal{supp}}
\def\F{{\cal F}}                 
\def\1{{\mathbf 1}}              
\def\Eset{{\cal E}}              
\def\Var{{\mathbf {Var}\,}}      
\def\supp{{\text{supp }}}        
\def\B{{\cal B}}                 
\def\P{{\mathbb P}}              

\begin{document}

\title{Multivariable averaging on sparse sets}
\author{P. LaVictoire, A. Parrish, J. Rosenblatt}
\maketitle
\begin{abstract} 
Nonstandard ergodic averages can be defined for a measure-preserving action of a group on a probability space, as a natural extension of classical (nonstandard) ergodic averages. We extend the one-dimensional theory, obtaining $L^1$ pointwise ergodic theorems for several kinds of nonstandard sparse group averages, with a special focus on the group $\Z^d$. Namely, we extend results for sparse block averages and sparse random averages to their analogues on virtually nilpotent groups, and extend Christ's result for sparse deterministic sequences to its analogue on $\Z^d$. The second and third results have two nontrivial variants on $\Z^d$: a ``native'' d-dimensional average and a ``product'' average from the 1-dimensional averages.
\end{abstract}
\tableofcontents
\section{\bf Introduction}
\label{intro}

\subsection{Pointwise ergodic theorems along sparse sets}
Subsequence ergodic theorems have been related to interesting questions since the beginnings of the subject: for instance, the equidistribution of $\{n^2\alpha\}$ for $\alpha\notin\Q$ corresponds to the mean convergence of the averages along the squares,

\begin{eqnarray}
\label{square}
\frac1N\sum_{k=1}^N f\circ T^{k^2}(x).
\end{eqnarray} 

The pointwise theory of nonstandard ergodic averages has an eventful history marked by a combination of ``qualitative'' methods, characterized by an interest in infinitary properties such as convergence,  and ``quantitative'' techniques, which spring from a more hard-analytical background. In the qualitative direction, Bellow and Losert \cite{BL} were the first to construct sparse sequences $\{n_k\}$ for which pointwise ergodic theorems could be proved; these sequences consist of increasingly large blocks of integers separated by increasingly, but not arbitrarily, vast gaps. The pointwise convergence (almost everywhere) of these averages followed from covering lemmas and asymptotic properties of the averaging sets.
\\
\\ A seminal quantitative result was Bourgain's proof \cite{JB2} that the averages along the squares (\ref{square}) converge pointwise (almost everywhere) as $N\to\infty$ for any dynamical system $(X,T)$ and any $f\in L^2(X)$. This result, and its many subsequent variants, used a transference argument and the Fourier-analytic properties of corresponding exponential sums to prove maximal and oscillational inequalities for the averages.
\\
\\ One distinction between the two methods concerns the endpoint space $L^1(X)$. The methods applied in \cite{BL} automatically prove convergence for all $f\in L^1$, but the Fourier transform methods for various sequences could be extended to $L^p$,  $p>1$, only by interpolation. As a result, the $L^1$ case for the sequence of squares remained open- for two decades, the sparsity of block sequences remained the only kind of sparsity for which $L^1$ pointwise ergodic theorems could be proved.
\\
\\In recent years, there has been substantial progress on the $L^1$ pointwise ergodic theory for sparse averages via a technique of Fefferman \cite{MR0257819} and Christ \cite{MC} which arises from the theory of singular integrals. Urban and Zienkiewicz \cite{UZ} used this method to prove an $L^1$ pointwise ergodic theorem for the averages along $\lfloor n^\alpha \rfloor$ (for $\alpha$ slightly greater than 1), and LaVictoire \cite{LaVic1} proved a similar pointwise theorem for random sequences: specifically that averages taken along a sequence of integers chosen randomly, and nearly as sparse as the sequence of square numbers, must converge pointwise almost surely for $f \in L^1$.
\\
\\However, Buczolich and Mauldin \cite{MR2680392} proved that the averages along the squares (\ref{square}) do \emph{not} satisfy an $L^1$ pointwise ergodic theorem; this argument is fundamentally a qualitative proof but uses some quantitative results from number theory. This result, among others, demonstrates that the $L^1$ case is distinct from, and requires different approaches compared to, the $L^p$ case, $p>1$.
\\
\\ Finally, in \cite{ChristPreprint}, Christ has provided new deterministic examples of sparse sequences that allow for pointwise convergence results in $L^1$. These sequences are constructed from the sets $\{(a,a^2,\dots,a^n)\subset \Z_p^n: a\in \Z_p\}$, using Freiman isomorphisms to find sequences of integers which have similar Fourier properties. The Weil bounds on the corresponding exponential sums give optimal Fourier bounds on the averages along such a sequence; these optimal Fourier bounds allow the application of another version of the technique of Fefferman and Christ, which obtains a weak $(1,1)$ maximal inequality.

\subsection{Pointwise ergodic theorems over discrete groups}
Another classically interesting extension of ergodic theory concerns more general averages obtained from a measure-preserving action $\cal T$ of a discrete group $G$ on a probability space $X$. To each $g\in G$ we associate a measure-preserving transformation ${\cal T}(g)$ on $X$, such that ${\cal T}(g)\circ {\cal T}(g')={\cal T}(g g')$. If $G=\Z^d$, in particular, a measure-preserving group action consists of $d$ commuting transformations $(T_1,\dots, T_d)$; for each $\vec n\in \Z^d$, we have the transformation ${\cal T}(\vec n)=T_1^{n_1}\cdots T_d^{n_d}$.
\\
\\Tempelman \cite{T} obtained a fairly general pointwise ergodic theorem (Theorem \ref{TempelmanThm}) for group actions, defined with respect to a sequence of finite subsets $F_N\subset G$. We consider the averages 
\begin{equation*}
A_Nf = \frac 1{\#F_N}\sum\limits_{\vec n \in F_N} f \circ {\cal T}(g)
\end{equation*}
 If the sequence of sets $\left\{F_N\right\}$ satisfies the Tempelman condition (\ref{TempelmanC}), then for any $f\in L^1(X)$, $A_Nf$ will converge pointwise for every measure-preserving action $\cal T$, irrespective of the underlying probability space $X$.\\
\\
There is a useful class of discrete groups which we will consider in particular:

\begin{definition}
A finitely generated group $G$ is \emph{virtually nilpotent} if it contains a nilpotent subgroup of finite index. 
\end{definition}

If we consider a finite set $\A$ of generators of a group $G$, and let $\mathbb{A}^N$ denote the words of length $\leq N$ in that alphabet, we have that these virtually nilpotent groups are precisely those in which $\# \A^N$ grows at a rate polynomial in $N$ (see \cite{Milnor}, \cite{Wolf}, and \cite{MR623534}). This polynomial growth rate allows for interesting applications of both qualitative and quantitative techniques in the cases of block sequences and random sequences, respectively, ultimately resulting in new pointwise convergence results.

\subsection{Sparse averages over discrete groups}
Combining the qualitative and quantitative approaches, we study sparse ergodic averages in the group setting and prove $L^1$ pointwise convergence results in a number of contexts.
\\
\\ By \textit{sparse}, we refer to a quality of a set, taken from a group $G$, analagous to that of being zero density in the integers. In the integers, density is determined by comparing a set to nested sets of closed intervals. When we consider virtually nilpotent groups, however, there are more choices for sets to which we can compare ours than in the case $G=\Z$. The sets defined by word-length in the generating set prove to be a good choice; by a result of Pansu \cite{Pan}, the corresponding notion of sparseness is independent of our choice of generating set.
\\
\\In addition to this basic idea of sparseness, we will also consider other descriptions of the density of a set. The first is an extension of the idea of upper Banach density; the second is an extension of the condition that gaps tend to $\infty$.

\begin{definition}
Let $F$ be a subset of a virtually nilpotent group $G$, and $\A$ a finite generating set. We say $F$ is \emph{sparse} if $$\lim_{N\to\infty}\frac{\#\left(F\cap \A^N\right)}{\#\A^N}=0.$$ We say that $F$ has \emph{Banach density $0$} if $$\lim_{N\to\infty} \sup_{g\in G} \frac{\#\left(gF\cap\A^N\right)}{\#\A^N}=0.$$ Finally, we say that $F$ has \emph{gaps tending to $\infty$} if, for every $N$, there are only finitely many $g\in F$ such that $g\A^N\cap F\neq\{g\}$.
\end{definition}

 We will exclude from our consideration sets that are sparse by virtue of essentially being averages over a subgroup of $G$- for instance the sets $\{1,\dots,N\}\times\{0\}\subset \Z^2$- since these are classical group averages in disguise. That is, we will be looking for sequences of sets $F_N$ such that if one takes any sequence of subsets $D_N \subset F_N$ with $\limsup\limits_{N \to \infty} \frac{\#D_N}{\#F_n} = 0$, then $\bigcup\limits_{N=1}^\infty F_n\backslash D_N$ still generates the entire group $G$.

\subsection{Results}
 We have three main categories of result, each obtained by a different method. Furthermore, in the case $G=\Z^d$, each of these has two variants. Colloquially, we call a construction \textit{plaid} if it consists of a Cartesian product of sparse subsets of $\Z$, and we call it \textit{native} if it is built in a genuinely $d$-dimensional manner. It should be noted that the results in the plaid case are still nontrivial, since the averages do not factor. (The associated Fourier transforms do factor, but this helps only in $L^2$ and not in $L^1$.)
\\

The first category of results, which are discussed in Section \ref{blocks} for $\Z^d$ and Section \ref{groupblock} for virtually nilpotent groups, extends the sparse block sequence result of Bellow and Losert \cite{BL} by showing that appropriate sparse sequences of sets satisfy the Tempelman Condition (\ref{TempelmanC}) and thus Tempelman's Ergodic Theorem (Theorem \ref{TempelmanThm}) applies. These sets are sparse, but do not have Banach density 0.
\\

The second category, developed in Sections \ref{random} for $\Z^d$ and \ref{nomad} for virtually nilpotent groups, extends the sparse random sequence result of LaVictoire \cite{LaVic1}. For simplicity, we will state here the speckled $\Z^d$ version only. Let $\{\xi_{\vec n} : \vec n \in\Z^d\}$ be independent $\{0,1\}$-valued random variables, such that if $2^j\leq |\vec n|< 2^{j+1}$, then $\P(\xi_{\vec n}=1)=2^{-\gamma j}\approx |\vec n|^{-\gamma}$.\\

Then for $\gamma<d/2$, with probability 1, the random set generated by $\{\xi_{\vec n}\}$ satisfies an $L^1$ pointwise ergodic theorem:
\begin{thm}\label{specseq}
For $\gamma<d/2$, the following holds for all $\omega\in\Omega$ except for a set of probability 0:
\\
Let $\{\vec a_k\}_{k\in\N}$ be an enumeration of the set $\{\vec n: \xi_{\vec n}(\omega)=1\}$ with $|\vec a_k|$ increasing. Then for any measure-preserving $\Z^d$-action $\cal T$ and any $f\in L^1(X)$, the averages
\begin{eqnarray*}
A_Nf(x):=\frac1N \sum_{k=1}^N f\circ {\cal T}(\vec a_k)(x)
\end{eqnarray*}
converge almost everywhere in $X$.
\end{thm}

Additionally, we prove a plaid version on $\Z^d$ (Corollary \ref{plaidseq}) and a version on general groups (Theorem \ref{L1}). The latter requires us to first prove an $L^2$ ergodic theorem (Theorem \ref{L2}) since the Fourier transform methods which work for $\Z^d$ do not work on nonabelian groups. An argument based in combinatorics and the $TT^*$ method from harmonic analysis suffices instead.
\\
\\ With probability 1, these random sequences are not only sparse, but have Banach density 0. Further, each can be modified to a sequence whose gaps tend to $\infty$ and along which an $L^1$ pointwise ergodic theorem still holds (Section \ref{gaps}).
\\
\\ The third category of result, developed in Section \ref{detsparse}, extends the sparse deterministic result of Christ \cite{ChristPreprint} to $\Z^d$ with both a native version (Theorem \ref{native}) and a product version (Theorem \ref{plaiddet}). We will state the simplest form of the first version. 
\begin{thm}\label{arith}
Let $p_k$ be prime numbers with $2^{k}<p_k<2^{k+\frac12}$. Then if we take the sparse set
\begin{eqnarray*}
S=\bigcup_{k=1}^\infty \bigcup_{j=0}^{p_k-1}(2^{k+1}+[j]_{p_k}, [j^2]_{p_k}, \cdots, [j^d]_{p_k})
\end{eqnarray*}
and order it by first coordinate, the averages along this set converge (almost everywhere) for every measure-preserving $\Z^d$-action $\cal T$ and any $f\in L^1(X)$.
\end{thm}
The a.e. convergence in this result is new even for the original averages in $\Z$; in \cite{ChristPreprint} only the weak $(1,1)$ maximal inequality is proved. The existence of a dense class for which pointwise convergence holds follows from an oscillational inequality (Theorem \ref{OscIneq}) for the averages, as in Section IV.2 of Rosenblatt and Wierdl's monograph, \cite{PETHA}. The pointwise convergence for $L^1$ functions follows. These sequences are sparse as well, and can be modified into a sequence whose gaps tend to $\infty$ without affecting the a.e. convergence of the averages.

\subsection*{Acknowledgments}
The authors thank M. Christ for substantial help on the results in Section 4, and for disseminating the preprint \cite{ChristPreprint}.

\section{\bf Averages Along Blocks}\label{blocks}
A sequence of sets $\{ F_n \}$ taken from a discrete group $G$ is called a F\o lner sequence if each set is finite and 
\begin{equation*}\label{FolnerCond}
 \lim_{n \rightarrow \infty} \frac{\# \left(gF_n \symdif F_n \right)}{\# F_n}  = 0
\end{equation*}
for every $g \in G$.
\\
\\ A F\o lner sequence is said to satisfy the Tempelman Condition if there is a constant $C$ so that 
\begin{equation}\label{TempelmanC}
 \#\left( F^{-1}_n F_n \right) \leq C \#F_n.
\end{equation}

\begin{definition}
Let $(X,\F,m)$ be a probability space and $\{T_g:g\in G\}$ a group of measure-preserving transformations on $X$ with $T_gT_h=T_{gh}$ for all $g,h\in G$.  We say that $\{T_g\}$ is a \emph{measure-preserving group action}.
\end{definition}

\begin{thm}[Tempelman's Ergodic Theorem, \cite{T}]\label{TempelmanThm}
Suppose that $T$ is a measure-preserving action of $G$ on the space $X$ and that $\{F_n \}$ is a nested F\o lner sequence that satisfies the Tempelman Condition. Then the averages
\begin{equation*}
 \frac{1}{\#F_n} \sum_{g \in F_n} f\left(T_gx\right)
\end{equation*}
converge for a.e. $x \in X$. 
\end{thm}

In this section we seek a generalized form of the block sequence example of Bellow and Losert \cite{BL}. We will examine both a plaid and a speckled version: our method, in each case, will be to show that a suitably chosen F\o lner sequence satisfies the Tempelman Condition, \ref{TempelmanC}. 
We note here that while each construction in this section will have zero density, neither are of Banach density zero: much like the original construction, the existence of the block-like structures does not allow for this type of sparseness. Further, in the first construction, we consider only two-variable free actions.  The ideas presented, however, will work in any finite number of variables.

\subsection{A Plaid Construction}
In this construction, our aim is to build a sparse sequence of sets by considering the products of one-dimensional sparse sequences. Among the difficulties in two dimensions is the loss of a natural order in which to take our averages; hence we must not only define our larger set, whence we derive sparseness, but also an ordering of its elements.  This ordering will give us our sequence of sets.

We start with a sequence of rectangles $\{D_k\}$ in $\mathbb Z^2$,
where each $D_k = R_k+(u_k,v_k)$ and $R_k = [1,a_k]\times [1,b_k]$.  We will need to have these rectangles well spaced,
so we will assume that
\begin{equation}\label{wellspaced}
u_{k+1} \ge u_k+a_k \,\,and\,\, v_{k+1} \ge v_k+b_k.
\end{equation}
We also want to arrange for these rectangles to provide
us projections along the axes that give sequences of zero density.  So we will also assume that
\begin{equation}\label{growingblocks}
\frac {\sum\limits_{i=1}^{k-1} a_i}{a_k}
\to 0 \mbox{ and } \frac {\sum\limits_{j=1}^{k-1} b_j}{b_k} \to 0 \mbox{ as } k\to \infty.
\end{equation}
It follows from Equation~\ref{growingblocks} that
$a_k \to \infty$ and $b_k \to \infty$ as $k \to \infty$.

An additional regularity assumption will be necessary in order to to prove almost everywhere convergence results.  It is that
for some constant $C > 0$, we have
\begin{equation}\label{regularity}
a_k \ge Cu_{k-1} \mbox{ and } b_k \ge Cv_{k-1}.
\end{equation}
This assumption is one of the conditions used by Bellow and Losert in \cite{BL}, and by Bellow, Jones, and Rosenblatt in \cite{BJR}.

We want to use this sequence of blocks to form unions of their projections on the coordinate axes,
and then form two variable averaging operators by putting these back together.
We denote the union of the projections of the blocks $\{D_1,\dots,D_k \}$ on the coordinate axes by $A(k) = \bigcup\limits_{i=1}^k [u_i+1,u_i+a_i]$, and
$B(k) =\bigcup\limits_{i=1}^k [v_i+1,v_i+b_i]$.  We want also to work with the intermediate blocks that come from the projections, so for
each $k$ and $r=1,\dots,a_k$, we denote by $A(k,r) = A(k-1) \cup [u_k+1,u_k+r]$, and $B(k,r)= B(k-1) \cup [v_k+1,v_k+r]$. Note that
$A(k,a_k) = A(k)$ and $B(k,b_k) = B(k)$.

The first fact to observe from \cite{BJR}, is that the sequences of sets $\{A(k)\}$ and $\{B(k)\}$ satisfy the Cone Condition: thus the operators
$\mathcal A_kf =\frac 1{a_k} \sum\limits_{i=1}^{a_k} f\circ S^{u_k+i}$ and
$\mathcal B_kf =\frac 1{b_k} \sum\limits_{j=1}^{b_k} f\circ T^{v_k+j}$ have maximal functions that are weak $(1,1)$ and strong $(p,p)$ for
all $p, 1 < p < \infty$.  Since  $a_k \to \infty$ and $b_k\to \infty$, we get pointwise and norm convergence to the projection of the invariant functions for all $L^p, 1\le p < \infty$ in each case.

Also, as implicit in \cite{BL} without proof, the intermediate sets $\{A(k,r): k\ge 1, 1\le r\le a_k \}$
and $\{ B(k,r): k\ge 1, 1\le r\le b_k \}$ are F\o lner sets that satisfy \ref{TempelmanC} (see Proposition~\ref{TempReg}).  In particular,
there is a constant such that
 $\#(A(k,r)-A(k,r)) \le C \#A(k,r)$ and $\#(B(k,r)-B(k,r)) \le C \#B(k,r)$ for all $k$ and $r$.
As a result, the operators
$\mathcal A(k,r)f =\frac 1{\#A(k,r)} \sum\limits_{i \in A(k,r)} f\circ S^i$ and
$\mathcal B(k,r)f =\frac 1{\#B(k,r)} \sum\limits_{j\in B(k,r)} f\circ T^j$ have maximal functions that are weak $(1,1)$ and strong $(p,p)$ for
all $p, 1 < p < \infty$.  (See Tempelman, \cite{T}).
So again we get pointwise and norm convergence to the projection of the invariant functions in each case, for all $L^p, 1\le p < \infty$.  The pointwise convergence result is clear from
 the Cone Condition if one were restricting
 oneself to the ends of the individual block (i.e. using only $A(k) = A(k,a_k)$ and $B(k) = B(k,b_k)$). 
 But this result actually requires the following computation even for
 these sequences, let alone the full sequence of intermediate sets. We give this proof because it does not appear in \cite{BL}
 and there are some not necessarily so obvious details that do need to be considered.

\begin{prop} \label{TempReg} 
The sequences of sets $\{A(k,r)\}$ and $\{B(k,r)\}$ satisfy the Tempelman's regularity conditions.
\end{prop}

\begin{proof}  We will consider only $\{A(k,r)\}$ since the argument for $\{B(k,r)\}$ is similar.
Our sets are all finite.  Also, any $A(k,r_1) \subset A(k+1,r_2)$ if $1\le r_1\le k$ and $1 \le r_2 \le k+1$,
 and $A(k,r_1)\subset A(k,r_2)$ if $1 \le r_1 \le r_2 \le k$.  Because these sets are unions of blocks $[u_i+1,u_i+a_i]$,
 and an intermediate block, with the lengths $a_i \to \infty$ as $i\to \infty$, it is clear that the sequence
 $(A(k,r))$ satisfies the F\o{}lner condition.  So the only condition remaining is
the fact that there is a constant M so that $\#( A(k,r) - A(k,r))  \le M \# A(k,r)$ for all $k$ and $r$.

Here we have
\begin{align*}
 A(k,r) - A(k,r) &=  \bigcup_{i=1}^{k-1} A_i \cup [u_k+1,u_k + r]-
               \bigcup_{i=1}^{k-1} A_i \cup [u_k+1,u_k + r] \\
 &= \left( \bigcup_{i=1}^{k-1} A_i - \bigcup_{i=1}^{k-1} A_i \right) \cup \left( \bigcup_{i=1}^{k-1} A_i - [u_k+1,u_k + r] \right)\\
 &\cup \left([u_k+1,u_k + r] -\bigcup_{i=1}^{k-1} A_i \right) \cup [u_k+1,u_k + r] - [u_k+1,u_k + r].
\end{align*}

So, as for the size of the set of differences, we are left with
\begin{align*}
 \# (A(k,r) - A(k,r)) &\leq  \#\left (\bigcup_{i=1}^{k-1} A_i - \bigcup_{i=1}^{k-1} A_I\right) +
 2 \# \left( \bigcup_{i=1}^{k-1} A_i - [u_k+1,u_k + r]\right) + 2r - 1.\\
\end{align*}

By Lemma 3.2 of \cite{BL}, we know that $\bigcup\limits_{i=1}^{k-1} A_i$ satisfies all of the requirements of Tempelman's Theorem.
In particular, this means there is then a constant $M_1$ so that
\begin{equation}\label{blockdifference}
 \#\left( \bigcup_{i=1}^{k-1} A_i - \bigcup_{i=1}^{k-1} A_i \right) \le M_1 \#\left( \bigcup_{i=1}^{k-1} A_i \right) = M_1 \sum_{i=1}^{k-1} a_i.
\end{equation}\\

Combining the assumptions above with our estimate, we then find that there is a constant $M_2$ so that
\begin{equation}\label{lk-1bound}
 \#\left ( \bigcup_{i=1}^{k-1} A_i - \bigcup_{i=1}^{k-1} A_i \right) \le M_2a_{k-1}.
\end{equation}\\

We now will turn our attention to $\#\left( \bigcup\limits_{i=1}^{k-1} A_i - [u_k+1,u_k + r] \right)$.\\

\begin{align*}
 &\bigcup_{i=1}^{k-1} A_i - [u_k+1,u_k + r]\\
 &=\left(\bigcup\limits_{i=1}^{k-2} A_i - [u_k+1, u_k + r] \right) \cup \left(A_{k-1} -[u_k+1,u_k + r] \right)\\
 &\subseteq [-u_k-r,\,u_{k-2} + a_{k-2}-u_k-1] \cup [u_{k-1}-u_k -r,\,u_{k-1} + a_{k-1} - u_k].\\
\end{align*}

Looking at the lengths of these intervals, we find that
\begin{equation*}
 \#\left( \bigcup_{i=1}^{k-1} A_i - [u_k+1,u_k + r] \right) \leq u_{k-2} + a_{k-2} + a_{k-1} + 2r +1.\\
\end{equation*}\\

But from our requirement relating the growth of $(u_k)$ and $(a_k )$, we know that there is a nonzero constant $M_3$ so that
\begin{equation*}
  u_{k-2} + a_{k-2} \leq M_3a_{k-1}.
\end{equation*}\\

So we have that
\begin{equation}\label{middle}
  \#\left( \bigcup_{i=1}^{k-1} A_i - [u_k,u_k + r] \right)
  \leq (M_3+1)a_{k-1} + 2r +1.
\end{equation}\\

Combining the inequalities \ref{lk-1bound} and \ref{middle}, we find that there are constants $M_4$ and $M_5$ so that
\begin{equation*}
\# \left( A(r,k) - A(r,k) \right) \leq M_4a_{k-1} + M_5r.
\end{equation*}\\

Now, suppose $r \leq a_{k-1}$. Then we have that
\begin{equation*}
\# \left( A(k,r) - A(k,r) \right) \leq \left( M_4+M_5 \right)a_{k-1}.
\end{equation*}

But $\#A(k,r)$ must be more than $a_{k-1}$ since $A(k,r)$ contains the ${k-1}$ block.  So
\[\# \left( A(k,r) - A(k,r) \right) \leq (M_4+M_5)\#A(k,r).\]

If, on the other hand, we have $r > a_{k-1}$, then
\begin{equation*}
 \#\left( A(k,r) - A(k,r) \right) < \left( M_4 + M_5 \right) r.
\end{equation*}

But $\# A(k,r) $ is more than $r$  because $A(k,r)$ contains the set
$[u_k+1,u_k + r]$.  So again
\[\# \left( A(k,r) - A(k,r) \right) \leq (M_4+M_5)\#A(k,r).\]

\end{proof}
\bigskip

We now want to put the sequences $\{A(k,r)\}$ and $\{B(k,r)\}$ back together.  Suppose we write $\bigcup\limits_{(k,r)}A(k,r) = \{s_m\}$
and $\bigcup\limits_{(k,r)} B(k,r) = \{t_n\}$ where $\mathbf s = \{s_m\}$ and $\mathbf t =\{t_n\}$ are increasing sequences.  Then take commuting maps $S$ and $T$
and consider the two variable averages
\[\mathcal A_{(M,N)}f = \frac 1{MN}\sum\limits_{m=1}^M\sum\limits_{n = 1}^Nf\circ S^{s_m}\circ T^{t_n}.\]

\begin{prop} The sequence $\left\{ \mathcal A_{(N,N)} \right\}$ has a maximal function that is weak $(1,1)$ and strong $(p,p)$.  Hence, for
an ergodic $\mathbb Z^2$ action, for all $f \in L^1(X)$,  we have $\lim\limits_{N \to \infty} \mathcal A_{(N,N)}f(x) = \int_X f\, dm$. 
\end{prop}

\begin{proof}  
Each $A_{(N,N)}$ corresponds to an average over a set of the form $A(k_1,r_1)\times B(k_2,r_2)$.
Because we have taken $M=N$ here, these sets are nested.  Also, these sets satisfy the F\o{}lner condition.
So, we obtain our result from Tempelman's Theorem \cite{T} if we have  a constant $C_o$ such that
 \[\#\left (A(k_1,r_1)\times B(k_2,r_2) - A(k_1,r_1)\times B(k_2,r_2)\right) \le C_o \#\left (A(k_1,r_1)\times B(k_2,r_2)\right ).\]
 But

\begin{eqnarray*}&&\#\left(A(k_1,r_1)\times B(k_2,r_2)-A(k_1,r_1)\times B(k_2,r_2)\right)\\
&&\,\,\,\,=\#\left((A(k_1,r_1)- A(k_1,r_1))\times (B(k_2,r_2)-B(k_2,r_2))\right) \\
&&\,\,\,\,=\#(A(k_1,r_1)- A(k_1,r_1))\,\#(B(k_2,r_2)-B(k_2,r_2))\\
&&\,\,\,\,\le C^2 \#A(k_1,r_1)\#B(k_2,r_2)\\
&&\,\,\,\,=C^2\#\left ( A(k_1,r_1)\times B(k_2,r_2)\right).
\end{eqnarray*}

\end{proof}

The result above is gives a two variable ergodic theorem with averaging over sets that have density zero along all horizontal and vertical lines in $\mathbb Z^2$. 

We would like to turn this block method into a sequence method as appears in Theorem~\ref{specseq} and Corollary~\ref{plaidseq}.
But we do not yet know how to amalgate enumerations of the supports of $\mathcal A_{(N,N)}$ to achieve this.

We would like to know more about integrability of the maximal function $\sup\limits_{N \ge 1} |\mathcal A_{(N,N)}f|$.  But first we
have this basic question.
\medskip
 
\begin{question}  
Do the maximal functions of $\mathcal A(k,r)f$ and $\mathcal B(k,r)$ map $L\log L$ to $L^1$?
\end{question}

It seems plausible that the answer to this question is affirmative because of the following fact:

\begin{prop}
The maximal functions for $\mathcal A_kf$ and $\mathcal B_kf$ map $L\log L$ to $L^1$.
\end{prop}

\begin{proof} 
We prove this result just for  $\mathcal A_k^*f$, the maximal function of $\mathcal A_kf$,because the proof for  $\mathcal B_kf$ is identical.
We refer to the notation and argument in the proof of Theorem 1, (a) in \cite{BJR}.
Here $C$ denotes an absolute constant, but not necessarily the same constant throughout.  Let $f^* = \sup\limits_{k\ge \infty}\frac 1{2k+1}\sum\limits_{j=-k}^{k} |f(T^jx)|$, a two sided version of the classical maximal function.
The classical result for $f \in L\log L$ gives the inequality $\|f^*\|_1 \le C\int_X |f|(1+\log^+ |f|)\, dm$.

Given the Cone Condition, there is a constant $C$ such that for all $\phi \in l_1(\mathbb Z)$,
 \[\#\{j: \mathcal A_k^*\phi(j) > 2\lambda\} \le C\,\#\{j: \phi^*(j) > \lambda\}.\]
This can be seen, with a slight
change of notation so as to be consistent with the notation here, by using the
$(B_i)$, defined on p. 45 in \cite{BJR}, and the inequality $\#\{j|M_{\Omega}\phi(j) > 2\lambda \}
\le C\,|\cup B_i|$, on p. 46 in \cite{BJR}.

By Calder\'on's transfer principle, this gives
\[m\{x:  \mathcal A_k^*f(x) > 2\lambda\} \le C\,m\{x: f^*(x) > \lambda\}.\]  

But then we have
\begin{eqnarray*} \|\mathcal A_k^*f\|_1 &\le& C\sum\limits_{n=1}^\infty   m\{x: \mathcal A_k^*f > 2n\}\\
&\le& C\, \sum\limits_{n=1}^\infty   m\{x: f^* > n\}\\
&\le& C\, \|f^*\|_1 \\
&\le& C\int_X |f|(1+\log^+ |f|)\, dm.
\end{eqnarray*}
\end{proof}

\begin{remark} 
More generally, one might ask how the maximal functions of $\mathcal A(k,r)f$ and $\mathcal B(k,r)f$
behave on general Orlicz spaces.  In particular, for which Orlicz functions $\Psi_1$ and $\Psi_2$ do these
maximal functions map $\Psi_1(L)$ to $\Psi_2(L)$?
\end{remark}

\begin{question} 
 Does the maximal function of $(A_{(N,N)}f)$ map $L\log^2 L$ to $L^1$?
This seems the correct choice of the domain for this result because the maximal function over the first variable
should be mapping $L\log^2 L$ to $L\log L$.
\end{question}

To obtain pointwise convergence results on $L^1(X)$, we needed to restrict our
two variable averages $\mathcal A(M,N)$ to just using  $\mathcal A_{(N,N)}f$.  While any nested sequence of rectangles would serve as well, there is a good reason to avoid a sequence of sets in which the side lengths are unrelated. 

\begin{prop} 
Suppose $S$ and $T$ commute and generate a free ergodic action.  Then
there exists a function $f \in L^1(X), f \ge 0$,
such that $\sup\limits_{(M,N)}A_{(M,N)}f = \infty$ a.e.
\end{prop}

\begin{proof}
Otherwise, by Sawyer's principle (see Sawyer, \cite{Sawyer}),
there is a weak inequality
\[m\{x\in X:\sup\limits_{(M,N)} |A_{(M,N)}f| >
\lambda\} \le \frac C{\lambda}\|f\|_1\]
for all $f \in L^1(X)$.
Now for $\phi \in \ell^1(\mathbb Z^2)$, let
\[A_{(M,N)}\phi(i,j)
=\frac 1{MN}\sum\limits_{m=1}^M\sum\limits_{n=1}^N\phi(s_m+i,t_n+j).\]
Because the action is a free ergodic action, our weak inequality transfers
to a weak inequality of the form
\[\#\{(i,j)\in \mathbb Z^2: \sup\limits_{(M,N)}|A_{(M,N)}\phi(i,j)| > \lambda\}
\le \frac C{\lambda}\|\phi\|_1\]
for all $\phi \in \ell^1(\mathbb Z^2)$.
Now take the function $\phi =\delta_{(0,0)}$.  We then
consider the set $E_n = \{(s_r,t_s) :rs\le n\}$.  Given
any $(i,j) \in E_n$, we have
$\sup\limits_{(M,N)} A_{(M,N)}\phi(i,j) \ge \frac 1{rs} \ge \frac 1n$.
Hence, with $\lambda = \frac 1n$, we have
\begin{eqnarray*}
CL &=&\frac C{1/n}\|\phi\|_1 \\
&\ge&  \#\{(i,j)\in \mathbb Z^2: \sup\limits_{M<N} A_{(M,N)}\phi(i,j) > \frac 1n\}\\
&\ge& \#E_n \\
&=&\#\{(s_r,t_s): rs \le n\}\\
&\ge& cn\log n.
\end{eqnarray*}
Letting $n$ tend to infinity gives a contradiction.
\end{proof}

\begin{remark} 
The divergence of the maximal function precludes there being a pointwise a.e. convergence
result on $L^1(X)$.  Also, notice here that it is not necessary to transfer the inequality,
just cleaner to state the idea. One could just work with large square
Rokhlin towers of height and width $n$ constructed within $X$ and use
$f = 1_B$, where $B$ is the base of the Rokhlin tower.   A similar
argument to the last string of inequalities will lead to a
contradiction as $n \to \infty$.
\end{remark}

Suppose now instead that we have an Orlicz space $\Psi(L)$ with
$\Psi$ some regular Orlicz function.
Both Stein \cite{Stein} and Sawyer \cite{Sawyer} give
useful results concerning maximal inequalities for Orlicz spaces.
We will need to have a regular
Orlicz function i.e. one such that $\|\Psi(|f|)\|_1$
and the norm $\|f\|_{\Psi(L)}$ are proportional to one
another.   Besides the usual properties of Orlicz functions,
regularity means that for some constant $K$, we have
$\Psi(2x) \le K\Psi(x)$ for all $x > 0$.

We can use the results of \cite{Sawyer} to prove the following.

\begin{prop} \label{allornone}
 Suppose $S$ and $T$ generate a free ergodic commuting
action.  If
\[\sup\limits_{(M,N)} |A_{(M,N)}f| < \infty\]
for all
$f \in \Psi(L)$, then there exists a constant $C$ such that
for all $f \in \Psi(L)$, we have for all $\lambda > 0$,
\[ m\{x: \sup\limits_{(M,N)} |A_{(M,N)}f| \ge \lambda\}
\le C \|\Psi(\frac 1{\lambda}|f|)\|_1.\]
Otherwise, for a residual set of functions $f \in \Psi(L)$, we have
$\sup\limits_{(M,N)} |A_{(M,N)}f| = \infty$ a.e.
\end{prop}

\begin{proof} 
See Theorems 3 and 4 in \cite{Sawyer}.
\end{proof}

This gives the following

\begin{prop} Suppose $S$ and $T$ generate a free ergodic commuting
action.  Suppose that $\Psi$ is a regular Orlicz function such that
$\Psi(x) = o(x\log^+ x)$ as $x \to \infty$.  Then there exists $f \in \Psi(L)$
such that $\sup\limits_{(M,N)} |A_{(M,N)}f| = \infty$ a.e.
\end{prop}
\begin{proof}  By Proposition~\ref{allornone}, if such an
$f$ did not exist, then we must have the weak inequality in
the proposition.  So take
the function $f = N1_B$ where $B$ is the base of a square Rokhlin tower
of height and width $n$ given by the free action determined by
$S$ and $T$.  We let $\lambda = \frac 1n$.  Then we have
\begin{eqnarray*} 
C\Psi(n) m(B) &=&C\|\Psi(L1_B)\|_1\\
&\ge& m\{x: \sup\limits_{(M,N)} A_{(M,N)}1_B \ge \frac 1n\}\\
&\ge& c n\log n m(B).
\end{eqnarray*}
Letting $n \to \infty$ gives a contradiction.
\end{proof}

\begin{remark}  
There is perhaps of generalization of the above to other
types of averaging besides Ces\`aro averaging.  Indeed, suppose $(\mu_M)$
and $(\nu_N)$ are two uniformly dissipative averaging methods on $\mathbb Z$.
Consider the two variable average $A_{(M,N)}f = \sum\limits_{m\in \mathbb Z}
\sum\limits_{n \in \mathbb Z} \mu_M(m)\nu_N(n) f\circ S^m\circ T^n$.   We
conjecture that for a free ergodic commuting action, there would exist a
function $f \ge 0$ such that $\sup\limits_{(M,N)} A_{(M,N)}f = \infty$.
a.e.  As for positive results in this generality, we do not have results
at this time.
\end{remark}

But in any case we do know that on $L\log L$ the averages $\mathcal A_{(M,N)}$ are 
well behaved.

\begin{prop}  On $L\log L$, and hence on any $L^p(X), 1 < p < \infty$, the averages
$\mathcal A_{(M,N)}f$ converge a.e. for all $f$.
\end{prop}

It is clear that the positive results above can be naturally extended to actions of $\mathbb Z^d$ with $d \ge 3$ too.
There will be analogous issues though on which Orlicz spaces are best to use when considering unrestricted mutlivariable
averages.

\subsection{A Divergent Construction}\label{speckledblock}

In the following section $B_r$ will denote the ball of radius $r$ in $\mathbb{Z}^d$, and $B^{+}_r$ will denote that part of the ball consisting of elements all of whose coordinate entries are positive.  Similarly,

\begin{equation*}
   \mathbb{Z}^{d}_+ = \left\{ (a_1, a_2, ..., a_d) \in \mathbb{Z}^d| \,\, a_i \geq 0 \mbox{ for all } i \geq 1 \right\}.
\end{equation*}

Let $I_k$ be a sequence of rectangular prisms in $\mathbb{Z}^{d}_+$ with one corner at the origin, each having diameter $\ell_k$.  We also require that the dimensions $b_1, b_2, ..., b_k$ of each prism satisfy 
\begin{equation*}
   c \leq \frac{b_i}{b_j} \leq C
\end{equation*}
where $c$ and $C$ are positive, absolute constants.  \\

Suppose these diameters satisfy
\begin{equation*}
   \frac{\sum_{i=1}^k \ell_i}{\ell_{k+1}} \rightarrow 0
\end{equation*}
as $k \rightarrow \infty$. \\

Further, suppose we have a sequence of vectors, $\vec{a}_k \in \mathbb{Z}^{d}_+$, so that 
\begin{align*}
   &\left| \vec{a}_{k+1} \right| > \left| \vec{a}_k \right| + \ell_k \mbox{, and}\\
   &\ell_k \geq C \left| \vec{a}_{k-1} \right|,
\end{align*}
where $C$ is some constant independent of $k$. \\

Let $S = \cup_{k} \left(\vec{a}_k + I_k \right)$.\\

We might like to add points to our average in an order depending only on their distance from the origin, i.e. use the sets $\left\{S \cap B_r : r \geq 1 \right\}$. However, since the ball of radius $r$ is very ``flat'' in the directions of the coordinate axes, we will either have to use a different ordering or place restrictions on the locations of blocks.\\

Take, for example, the set in $\mathbb{Z}^2$, $S = \cup_{k} \left(\vec{a}_k + I_k \right)$, where, in addition to the requirements above, we have that 
\begin{enumerate}
 \item $I_k$ are squares with diameters $\ell_k$, \label{1}
 \item $\ell^{2}_k \leq |\vec{a}_k|$, \label{2}
 \item $\frac{\sqrt{2}}{2} \ell_{k+1} >  \left( \sum_{i=1}^k \ell^{2}_i \right)^2$, and \label{3}
 \item $ \vec{a}_k =  (a_k, 0).$ \label{4}
\end{enumerate}

Define the subsequence $\left\{ S \cap B_{r_k} : r_k = \sqrt{|\vec{a}_k|^2 + \frac{1}{2} \ell^{2}_k} \right\}$. This sequence consists of sets comprised of all blocks before the $k$th, and the left face of the $k$th block.\\

Define $f_k(x,y): \mathbb{Z}^2 \rightarrow \mathbb{R}$ by
\begin{equation*}
   f_k(x,y) = \begin{cases}
                 1 & \text{ for } x=a_k, \,\, 0 \leq y \leq \frac{\sqrt2}{2} \ell_k\\	
		 0 & \text{ otherwise.}
              \end{cases}
\end{equation*}
and let $f(x,y) = \sum_{k} f_k(x,y)$.

We then have that the average over $\left\{ S \cap B_{r_k} \right\}$,
\begin{equation*}
   \frac{1}{\# S \cap B_{r_k}} \sum_{(x,y) \in S \cap B_{r_k}} f(x,y) >  \frac{\frac{\sqrt2}{2}\ell_k}{\left(\frac{1}{2} \sum_{i=1}^{k-1} \ell^{2}_i \right) + \frac{\sqrt2}{2}\ell_k}
\end{equation*}
while the average taken at the ends of each block
\begin{equation*}
      \frac{1}{\# S \cap B_{r_s}} \sum_{(x,y) \in S \cap B_{r_s}} f(x,y) =\frac{\sqrt2 \sum_{i=1}^k \ell_i}{ \sum_{i=1}^k \ell^{2}_i}.
\end{equation*}

The first average is larger than $1/2$ while the second tends to $0$; each may be transferred to our measure preserving system. This answers in the affirmative the question of whether there is a sequence of sets for which a subsequence converges while the larger sequence diverges.\\

This example illustrates the interaction between the metric of the group and the measure of the various blocks; it works becuase the additional requirements above allow us to slice off a single face of the block at a time, and the growth requirement allows for these faces to outweigh the measure of the previous blocks. \\

Removing either of these conditions results in a pointwise theorem.  For example, suppose we required our $\vec{a}_k$ to not lie along an axis.  Suppose that our set $S$ is constructed in accord with the earlier requirements alone; that is without conditions \ref{1} through \ref{4}.  We will require instead that
\begin{equation}\label{cornersfirst}
   \vec{a}_k = (a_k, a_k, ..., a_k).
\end{equation}

\begin{prop} 
With requirement (\ref{cornersfirst}), $\left\{S \cap B_r: r \geq 1 \right\}$ is a pointwise $L^1$-good sequence of sets for any aperiodic $\mathbb{Z}^d$ action.
\end{prop}

\begin{proof}
As in Bellow and Losert's proof for the block sequence, we will rely on the Tempelman Ergodic Theorem to complete the proof.  
The first three requirements are filled- the sets are all of finite volume, each is nested within the next, and the F\o lner property is satisfied through the lacunary growth of the diameters $\ell_k$.  All that remains is the difference requirement. \\

Here we divide our set $S \cap B_r$ into two: a $S_m = \cup_{k=1}^{m} \left(\vec{a}_k + I_k \right)$, including everything up to the last complete prism in $S$, and a remainder $R$, consisting of everything else. The crucial difference between this and the earlier example is that the remainder $R$ must be either trivial or properly $d$-dimensional. \\

Since
\begin{equation*}
   S \cap B_r = S_m \cup R,
\end{equation*}
we have
\begin{align}
   \# \left( S \cap B_r - S \cap B_r \right) &\leq \# \left( S_m - S_m \right) + \# \left( R - S_m \right) + \# \left( S_m -R\right)+ \# \left( R- R\right) \notag \\
&= \# \left( S_m - S_m \right) + 2\# \left( R - S_m \right) + \# \left( R- R\right). \label{difference}
\end{align}

We will deal with the first term first, noting that this would complete the proof of our claim were we only interested in the subsequence consisting of sets with only complete prisms. \\

In a similar fashion to the decomposition in Section \ref{groupblock}, we may write 
\begin{equation*}
   S_m = \left(\vec{a}_m + I_m\right) \cup S_{m-1}.
\end{equation*}

We then note that 
\begin{align*}
   \# \left( S_m - S_m  \right) \leq \# \left\{\left(\vec{a}_m + I_m \right) -\left(\vec{a}_m + I_m \right) \right\} + 2\# \left\{ \left(\vec{a}_m + I_m \right) - S_{m-1} \right\} + \# \left\{ S_{m-1} - S_{m-1} \right\}.
\end{align*}

For the first term, we have 
\begin{align*}
   \# \left\{\left(\vec{a}_m + I_m \right) -\left(\vec{a}_m + I_m \right) \right\} &= \# \left(I_m - I_m \right) 
  &\leq \# \left( B_{\ell_m} - B_{\ell_m} \right)
  &\leq \# B_{\ell_m} \leq C \# I_m,
\end{align*}

For the third term, we have that $S_{m-1} \subset B_{\left|\vec{a}_{m-1}\right| + \ell_{m-1}}$, thus the difference 
\begin{equation*}
  S_{m-1} - S_{m-1} \subset B_{\left|\vec{a}_{m-1}\right| + \ell_{m-1}} - B_{\left|\vec{a}_{m-1}\right| + \ell_{m-1}},
\end{equation*}
which, in turn, is contained in $B_{\left|\vec{a}_{m-1}\right| + \ell_{m-1}}$.  The volume of this ball is a constant multiple of $\left( \left|\vec{a}_{m-1}\right| + \ell_{m-1} \right)^d$.  But by our second condition on the spacing sequence above, 
\begin{equation*}
   \left( \left|\vec{a}_{m-1}\right| + \ell_{m-1} \right)^d \leq \left( \frac{1}{C} \ell_m + \ell_{m-1} \right)^d,
\end{equation*}
whence we have that this difference is less than $C^d \ell^{d}_m$ for some constant $C$. 

For the second difference, we note that shifting a set in $\mathbb{Z}^d$ does not change its measure.  So,
\begin{align*}
   \# \left\{ \left(\vec{a}_m + I_m \right) - S_{m-1} \right\} &= \# \left\{ \vec{a}_m + I_m - S_{m-1} \right\}\\
&= \# \left\{ -\vec{a}_m + \vec{a}_m + I_m  - S_{m-1} \right\}
&< \# \left\{ B_{\ell_m} -  B_{\left|\vec{a}_{m-1}\right| + \ell_{m-1}} \right\}
&< \# B_{\left|\vec{a}_{m-1}\right| + \ell_{m-1}},
\end{align*}
which, as before, has a volume less than  $C^d \ell^{d}_m$ for some constant $C$.

So we have that the difference $\# \left( S_m - S_m  \right) \leq C \# S_m$. \\

The size of the second and third terms in \ref{difference} are dependent upon the diameter of $R$.  Defining $s$ to be the diameter of $R$, we note that $s \leq \ell_{m+1}$ and the measure of $\# R \geq Cs^d$, by our orignial conditions on the block $I_{m+1}$ and requirement (\ref{cornersfirst}).\\

For the third term, we note that
\begin{equation*}
   R - R \subseteq \left( \vec{a}_{m+1} + B_{s} \right) -\left( \vec{a}_{m+1} + B_{s} \right).
\end{equation*}
Thus we have that 
\begin{align*}
  \# \left(  R - R \right) &\leq \# \left( B_s - B_s \right) \\
&= \# B_s\\
&\leq Cs^d \leq c \# R,
\end{align*}
where $C$ and $c$ are different constants. 

This leaves us with only the second term.  Here we have 
\begin{align*}
   \# \left( R - S_m \right) &\leq \# \left(B_s - B_{c \ell_m} \right)\\
&\leq C \max( s^d, (c\ell_m)^d ) \leq C \# \left( S \cap B_r \right)
\end{align*}
   for constants $c$ and $C$. 

\end{proof}

\begin{remark}
   Taking another approach, we might instead place an additional requirement on the diameters, $\ell_k$, relative to the spacing vectors, $\vec{a}_k$, in order to insure that the first set which might outweigh the previous blocks would necessarily be $d$-dimensional.
\end{remark}

\section{Multivariable Random Averages}\label{random}

There are two natural ways that we might extend the results of \cite{LaVic1} to $\Z^d$ actions, depending on the choice of random variables corresponding to $\vec{n}$. The first, which hews to the product structure of $\Z^d$, is to write $\xi_{\vec n}=\prod_{i=1}^d \xi_{i,n_i}$, where $\{\xi_{i,n}:1\leq i\leq d; n\in\N\}$ are independent random variables. The other is to simply take independent random variables $\{\xi_{\vec n}: \vec n \in\Z^d\}$.
\\
\\ These different approaches can be characterized as the ``plaid'' and ``speckled'' approaches, respectively, according to the patterns of points that they select in $\Z^2$. The relevant distinction, for us, will concern the difference set: if we let $S_N^\omega$ denote the set of all $\vec n$ with $\xi_{\vec n}=1$ and $|\vec n|\leq N$, then in the speckled case, for any $\vec k\neq0$ there will be (with overwhelming probability) not too many ways to write $\vec k$ as the difference of two elements of $S_N^\omega$ (and this number of representations will in fact be very close to its probabilistic mean); but in the plaid case, there will be significantly more representations whenever any of the components of $\vec k$ are zero. For this reason, we will first prove the result in the simpler speckled case.

\begin{remark}
We may also consider the ``plaid diagonal'' pattern obtained by taking $\xi_{i,n}=\xi_{1,n}$ for all $i$; however, this does not exhibit any behavior different from the first plaid case, and is more difficult to calculate.
\end{remark}

\begin{remark}
One might hope that the product theory might help us to prove the maximal inequality for the plaid case directly from the one-dimensional result; however, this is not the case in $L^1$, any more than it is for the Bellow-Losert construction.
\end{remark}

\begin{remark}
Our reliance on the ``uniformity'' of the ways to represent points as elements of $S_N^\omega-S_N^\omega$ gives us a natural bound on how sparse a random set we could expect the technique to work for: namely, the average number of representations of a point $\vec k$ with $|\vec k|\lesssim N$ should tend to infinity as $N\to\infty$, and thus we will need $\#S_N^\omega\gg\sqrt{\# B(0,N)}\approx N^{d/2}$. Indeed, we will have a result precisely when our random sets have $\#S_N^\omega\approx N^\alpha$ with $\alpha>d/2$.
\end{remark}

\subsection{Speckled Random Averages on $\Z^d$}
Using the Fourier transform, it is an immediate extension of the one-dimensional theory that both the plaid and the speckled sequences are universally $L^2$-good.\\
\\
Thus we need only prove a weak maximal inequality on $L^1$, and (by the positivity of the averaging operators) this only on dyadically increasing sets. Finally, by the Calder\'{o}n transference principle (Lemma \ref{transfer}), we can instead prove this maximal theorem for the corresponding convolution operators on $\Z^d$. That is, we let $\{\xi_{\vec n} : \vec n \in\Z^d\}$ be independent $\{0,1\}$-valued random variables, such that if $2^j\leq |\vec n|< 2^{j+1}$, then $\P(\xi_{\vec n}=1)=2^{-\gamma j}\approx |\vec n|^{-\gamma}$. Then define
\begin{eqnarray}
\mu_j^{(\omega)}(\vec n)&:=&\left\{\begin{array}{ll} 2^{(\gamma-d)j}\xi_{\vec n}(\omega), & 2^j\leq |\vec n|< 2^{j+1} \\ 0 & \text{otherwise;}\end{array}\right.\\
\nu_j^{(\omega)}(\vec n)&:=&\left\{\begin{array}{ll} \mu_j^{(\omega)}(\vec n)-2^{-d j}, & 2^j\leq |\vec n|< 2^{j+1} \\ 0 & \text{otherwise.}\end{array}\right.
\end{eqnarray}
Then the random convolution operators $f\to f\ast\mu_j^{(\omega)}$ correspond to our random ergodic averages, and the $\nu_j^{(\omega)}$ are mean 0 variants of the same. The role of the difference set, alluded to previously, is expressed in terms of the convolution of the mean 0 measure $\nu_j^{(\omega)}$ with its reflection $\tilde\nu_j^{(\omega)}(\vec v)=\nu_j^{(\omega)}(-\vec v)$, which we will use later in the argument. Clearly $\nu_j^{(\omega)}\ast\tilde\nu_j^{(\omega)}(0)=\|\nu_j^{(\omega)}\|_{\ell^2(\Z^d)}^2$, but at all nonzero points the convolution should be small:
\begin{lem}
\label{speckled}
Let $\epsilon>0$. With probability 1 in $\Omega$, there exists $C_{\omega,\epsilon}<\infty$ such that
\begin{eqnarray}
\label{Snon0}
\|\nu_j^{(\omega)}\ast\tilde\nu_j^{(\omega)}\|_{\ell^\infty(\Z^d_\times)}\leq C_{\omega,\epsilon} 2^{(\gamma-\frac{3d}2 +\epsilon)j},
\end{eqnarray}
where $\Z^d_\times:=\Z^d\setminus\{0\}$. Furthermore,
\begin{eqnarray}
\label{S0}
|\nu_j^{(\omega)}\ast\tilde\nu_j^{(\omega)}(0)|\leq C_{\omega,\epsilon} 2^{(\gamma-d)j}.
\end{eqnarray}
\end{lem}
This follows from the more general Lemma \ref{cancels} and the Borel-Cantelli Lemma.\\
\\
We must now prove the following:
\begin{thm}
\label{SgtPepper}
For all $\gamma<d/2$, there is a set of probability 1 in $\Omega$ such that for each $\omega$ in this set, there exists $C_\omega<\infty$ such that for every $\lambda>0$,
\begin{eqnarray}
\#\{\vec n: \sup_j | f\ast\mu_{j}^{(\omega)}(\vec n)|>\lambda\}\leq \frac{C_\omega}{\lambda}\|f\|_1.
\end{eqnarray}
\end{thm}
This immediately implies Theorem \ref{specseq}, which gives multivariable averaging methods over sets of Banach density zero.\\
\\
(We will suppress the superscripts $(\omega)$ in the rest of this section, and presume that $\nu_j$ are measures which satisfy the bounds in Lemma \ref{speckled} for a constant $C$.)
\begin{proof}
The method of this proof begins with Calder\'{o}n-Zygmund theory. We decompose our function $f$ in the standard fashion, and imitate the technique of Christ \cite{MC} for operators on $\R^d$ whose kernels are not properly differentiable. For the crucial parts, $B$, of our decomposition of $f$, we consider the inner products $\langle \nu_j\ast B,\nu_j\ast B\rangle=\langle \tilde\nu_j\ast\nu_j\ast B,B\rangle$ and use the extra cancellation of that convolution to obtain an especially strong $L^2$ bound. This is necessary in order to compensate for the waste of using an $L^2$ bound to prove an $L^1$ bound. This technique was first applied to pointwise ergodic theorems in $L^1$ by Urban and Zienkiewicz \cite{UZ}.
\\
\\ By scaling, we may assume that $\lambda\approx1$. We will choose $\lambda$ to be a bit larger than 1, but our choice will depend only on $d$ and $\omega$, not on $f$ or on $j$. We apply the standard discrete Calder\'{o}n-Zygmund decomposition at height 1, giving us $f=g+b=g+\sum_{s,k}b_{s,k}$, where $\|g\|_\infty\leq1$, each $b_{s,k}$ is supported on a discrete dyadic cube $Q_{s,k}$ with side length $2^s$, $\|b_{s,k}\|_1\leq 2^{d s}$, and $\sum_{s,k} |Q_{s,k}|\leq 2^d\|f\|_1$. We can also require each $b_{s,k}$ to have mean 0, but we shall not need this.
\\
\\ Now we will use to our advantage several key properties of $\mu_j$ and $\nu_j$ to reduce this problem to its final form. First, we have normalized $\mu_j$ so that with probability 1 in $\Omega$, there is a $C_\omega$ such that $\|\mu_j\|_1\leq C_\omega$ for all $j$, and thus $\|g\ast\mu_j\|_\infty\leq C_\omega$ for all $j$. Therefore, if $\sup_j | f\ast\mu_{j}(\vec n)|>\lambda$, then $\sup_j |\sum_{s,k} b_{s,k}\ast\mu_{j}(\vec n)|>\lambda- C_\omega$.
\\
\\We next split $b_s:=\sum_k b_{s,k}$ into two pieces depending on its size. Define
\begin{eqnarray}
\label{Stall}
b^{(j)}_s(\vec n)&:=&\left\{\begin{array}{ll} b_s(\vec n), & |b_s(\vec n)|>2^{(d-\gamma)j} \\ 0 & \text{otherwise;}\end{array}\right.\\
\label{Sshort}
B^{(j)}_s&:=& b_s-b^{(j)}_s.
\end{eqnarray}
Note that the threshold $2^{(d-\gamma)j}$ is proportional to $|\supp \mu_j|$. This allows us to prove a bound on the support of $\sup_j |\sum_s b^{(j)}_s\ast\mu_{j}(\vec n)|$, since $b_{s,k}$ cannot be large on a large set. In particular, if we let $b^{(j)}:=\sum_s b^{(j)}_s$,
\begin{eqnarray*}
\#\{\vec n: \sup_j |\sum_s b^{(j)}_s\ast\mu_{j}(\vec n)|>0\}&\leq&\sum_j\#\{\vec n: |b^{(j)}\ast\mu_{j}(\vec n)|>0\}\\
&\leq& \sum_j |\supp \mu_j|\cdot\#\{\vec n: |b(\vec n)|>2^{(d-\gamma)j}\}\\
&\leq& \sum_j C_\omega 2^{(d-\gamma)j}\sum_{i\geq j} \#\{\vec n: 2^{(d-\gamma)i}<|b(\vec n)|\leq 2^{(d-\gamma)(i+1)}\}\\
&=& C_\omega \sum_i \#\{\vec n: 2^{(d-\gamma)i}<|b(\vec n)|\leq 2^{(d-\gamma)(i+1)}\} \sum_{j\leq i} 2^{(d-\gamma)j}\\
&\leq& C_\omega \sum_i \#\{\vec n: 2^{(d-\gamma)i}<|b(\vec n)|\leq 2^{(d-\gamma)(i+1)}\} \cdot C 2^{(d-\gamma)i}\\
&\leq& C'_\omega \|b\|_1\leq C''_\omega \|f\|_1
\end{eqnarray*}
because the second-to-last line is a lower sum for $b$. This is an acceptable bound, so we need only consider the contribution of the $B^{(j)}_s$, which have the property $\|B^{(j)}_s\|_\infty\leq 2^{(d-\gamma)j}$. This will be precisely what we need to control the contribution of $\nu_j\ast\tilde\nu_j(0)$.
\\
\\Since $\mu_j-\nu_j$ is simply an appropriately normalized average over all points of magnitude between $2^j$ and $2^{j+1}$, the weak (1,1) bound on the standard maximal function tells us that
\begin{eqnarray*}
\#\{\vec n: \sup_j |\sum_s B^{(j)}_s\ast(\mu_{j}-\nu_j)(\vec n)|>1\}&\leq& \#\{\vec n: \sup_j |b|\ast|\mu_{j}-\nu_j|(\vec n)>1\}\\
&\leq& C\|b\|_1\leq C'\|f\|_1
\end{eqnarray*}
and we may indeed replace $\mu_j$ with $\nu_j$.
\\
\\ Finally, if we let $Q^*_{s,k}$ denote the points that are within distance $2^s$ of the cube $Q_{s,k}$, then $b_{s,k}\ast\nu_j$ is supported in $Q^*_{s,k}$ if $j<s$. We take the exceptional set $E$ to be the union of all $Q^*_{s,k}$, and note that $|E|\leq C_d\|f\|_1$. Since $E$ is of acceptable size and $\supp B^{(j)}_s\ast\nu_j\subset E$ for all $s>j$, we only need to bound the set $$\#\{\vec n: \sup_j |\sum_{s\leq j} B^{(j)}_s\ast\nu_{j}(\vec n)|>1\}.$$ This we will do using an $L^2$ bound.
\begin{remark}
Most of the above reductions are standard, with the exception of splitting $b_{s,k}=b^{(j)}_{s,k}+B^{(j)}_{s,k}$ by height. This is done because $\nu_j\ast\tilde\nu_j(0)$ has no cancellation, and thus we must treat separately the contribution of a single delta mass in our inner product. The condition $\|B^{(j)}_{s,k}\|_\infty\leq 2^{(d-\gamma)j}$ precisely suffices to balance the trivial bound (\ref{S0}) in these terms. Note that there is no equivalent for this in the original problems on $\R^d$, since there the convolutions have a singularity rather than a delta mass at the origin; because of this, the analogous endpoint theorems require $f$ in the Hardy space $H^1$ rather than $L^1$.
\end{remark}
By Chebyshev's Inequality,
\begin{eqnarray*}
\#\{\vec n: \sup_j |\sum_{s\leq j}B^{(j)}_{s}\ast\nu_{j}(\vec n)|>1\}&\leq& \left\|\sup_j\left|\sum_{s\leq j} B^{(j)}_{s}\ast\nu_j\right| \right\|_2^2\\
&\leq&\sum_j\left\|\sum_{s\leq j} B^{(j)}_{s}\ast\nu_j\right\|_2^2\\
&\leq&2\sum_j\sum_{\substack{s,t: \\ s\leq t\leq j}}\left\langle B^{(j)}_{s}\ast\nu_j, B^{(j)}_{t}\ast\nu_j\right\rangle\\
&=&2\sum_j\sum_{\substack{s,t: \\ s\leq t\leq j}}\left\langle B^{(j)}_{s}\ast\nu_j\ast\tilde\nu_j, B^{(j)}_{t}\right\rangle.
\end{eqnarray*}
First we assume that $\supp B^{(j)}_s\subset Q_{j,k_1}$ and $\supp B^{(j)}_t\subset Q_{j,k_2}$, where each of these is a single dyadic cube of size $2^j$. Then it is easy to see that $\|B^{(j)}_{s}\|_1\leq 2^{dj}$, and thus by Lemma \ref{speckled},
\begin{eqnarray*}
\left |\left\langle B^{(j)}_{s}\ast\nu_j\ast\tilde\nu_j, B^{(j)}_{t}\right\rangle\right|&\leq& \|\nu_j\ast\tilde\nu_j\|_{\ell^\infty(\Z^d_\times)}\|B^{(j)}_{s}\|_1\|B^{(j)}_{t}\|_1+\nu_j\ast\tilde\nu_j(0)\left\langle B^{(j)}_{s},B^{(j)}_{t}\right\rangle\\
&\leq& C_{\omega,\epsilon} 2^{(\gamma-3d/2 +\epsilon)j}\|B^{(j)}_{s}\|_1\|B^{(j)}_{t}\|_1+ C_{\omega,\epsilon} 2^{(\gamma-d)j}\delta_{s=t}\|B^{(j)}_{s}\|_\infty\|B^{(j)}_{t}\|_1\\
&\leq& C_{\omega,\epsilon} 2^{(\gamma-d/2 +\epsilon)j}\|B^{(j)}_{t}\|_1+ C_{\omega,\epsilon} \delta_{s=t}\|B^{(j)}_{t}\|_1.
\end{eqnarray*}
We can remove the restriction on the supports since the inner product is 0 whenever the support of $B^{(j)}_s$ and the support of $B^{(j)}_t$ are separated by at least $2^{j+1}$; thus the double sum over all $Q_{j,k_1}$ and $Q_{j,k_2}$ is, up to a fixed constant, a single sum. Therefore we have
\begin{eqnarray*}
\#\{\vec n: \sup_j |\sum_{s\leq j}B^{(j)}_{s}\ast\nu_{j}(\vec n)|>1\}&\lesssim& \sum_j\sum_{\substack{s,t: \\ s\leq t\leq j}}2^{(\gamma-d/2 +\epsilon)j}\|B^{(j)}_{t}\|_1+ \delta_{s=t}\|B^{(j)}_{t}\|_1\\
&\leq &\sum_j\sum_t (j2^{(\gamma-d/2 +\epsilon)j}+1)\|B^{(j)}_{t}\|_1
\end{eqnarray*}
and clearly, this is $\leq\|b\|_1\leq C\|f\|_1$ so long as $\gamma<d/2 -\epsilon$. Since $\epsilon>0$ is arbitrary, we have proved Theorem \ref{SgtPepper}.
\end{proof}

\subsection{Plaid Random Averages on $\Z^d$}
In this section, we take independent random variables $\{\xi_{i,n}:1\leq i\leq d; n\in\N\}$ with $\P(\xi_{i,n}=1)=n^{-\alpha}$. We define
\begin{eqnarray}
\mu_j^{(\omega)}(\vec n)&:=&\left\{\begin{array}{ll} 2^{d(\alpha-1)j}\prod_{i=1}^d\xi_{i,n_i}(\omega), & 2^j\leq |\vec n|< 2^{j+1},\; n_i>0 \; \, \text {for all}\, i \\ 0 & \text{otherwise;}\end{array}\right.\\
\nu_j^{(\omega)}(\vec n)&:=&\left\{\begin{array}{ll} \mu_j^{(\omega)}(\vec n)-\prod_{i=1}^d n_i^{-\alpha}, & 2^j\leq |\vec n|< 2^{j+1},\; n_i>0 \; \, \text {for all}\, i \\ 0 & \text{otherwise.}\end{array}\right.
\end{eqnarray}
As mentioned before, our decomposition of the random measure $\nu_j^{(\omega)}\ast\tilde\nu_j^{(\omega)}$ will be more complicated because the additional structure removes part of the cancellation at points where some of the coordinates are zero.
\begin{lem}
\label{plaid}
Let $\epsilon>0$. With probability 1 in $\Omega$, there exists $C_{\omega,\epsilon}<\infty$ such that
\begin{eqnarray*}
\nu_j^{(\omega)}\ast\tilde\nu_j^{(\omega)}=\sum_{I\subset\{1,\dots,d\}}\chi_{j,I}^{(\omega)}
\end{eqnarray*}
where each $\chi_{j,I}^{(\omega)}$ is supported on $\{\vec n: n_i\neq0\; \, \mbox{ for all } i\in I\}$, and for all $I\neq\emptyset$,
\begin{eqnarray}
\label{P}
\|\chi_{j,I}^{(\omega)}\|_\infty\leq C_{\omega,\epsilon} 2^{(-d-\frac{|I|}2+d\alpha+\epsilon)j}.
\end{eqnarray}
Furthermore,
\begin{eqnarray}
\label{P0}
|\nu_j^{(\omega)}\ast\tilde\nu_j^{(\omega)}(0)|\leq C_{\omega,\epsilon} 2^{d(\alpha-1)j}.
\end{eqnarray}
\end{lem}
\begin{remark}
We could aggregate all of the $\chi_{j,I}^{(\omega)}$ with $I\neq\emptyset$ into a single function, and for $\alpha<\frac1{2d}$ the overall $\ell^\infty$ bound would be good enough to prove the weak maximal inequality just as we proved Theorem \ref{SgtPepper}. However, with a more targeted Calder\'{o}n-Zygmund decomposition, we can do better, and obtain the weak maximal inequality for $\alpha<1/2$. Thus we can prove an $L^1$ pointwise ergodic theorem for plaid sets which are just as sparse as the speckled sets we have proved it for.
\end{remark}
\begin{thm}
\label{MMT}
For all $\alpha<1/2$, there is a set of probability 1 in $\Omega$ such that for each $\omega$ in this set, there exists $C_\omega<\infty$ such that for every $\lambda>0$,
\begin{eqnarray}
\#\{\vec n: \sup_j | f\ast\mu_{j}^{(\omega)}(\vec n)|>\lambda\}\leq \frac{C_\omega}{\lambda}\|f\|_1.
\end{eqnarray}
\end{thm}
\begin{cor}\label{plaidseq}
For all $\alpha<1/2$, there is a set of probability 1 in $\Omega$ such that every $\omega$ in this set has the following property: Let $\{\vec a_k :k\in\N\}$ be an enumeration of the set $\{\vec n: \xi_{i,n_i}(\omega)=1 \,\, \text{for all}\,\, i\}$ with $|\vec a_k|$ increasing. Then for any measure-preserving $\Z^d$-action $\cal T$ and any $f\in L^1(X)$, the averages
\begin{eqnarray*}
A_Nf(x):=\frac1N \sum_{k=1}^N f\circ {\cal T}(\vec a_k)(x)
\end{eqnarray*}
converge a.e. in $X$.
\end{cor}
\begin{proof}
The proof of Theorem \ref{MMT} follows the same lines as that of Theorem \ref{SgtPepper}, with one substitution: we will have to do more when we split $b_{s,k}=b^{(j)}_{s,k}+B^{(j)}_{s,k}$, so that we can get better bounds on the pieces of the inner product which correspond to the various $\chi_{j,I}$.
\\
\\Let us jump ahead to that inner product to determine the bounds we will need; again we will assume that each of $B^{(j)}_s$ and $B^{(j)}_t$ is supported on a dyadic cube of sidelength $2^j$. Let $\vec n_I\in\Z^{\# I}$ denote the projection of $\vec n$ onto the coordinates indexed by $I$, and $\vec n_I^\perp\in\Z^{d-\#I}$ the projection onto the other $d-\#I$ coordinates. Then
\begin{eqnarray*}
\left |\left\langle B^{(j)}_{s}\ast\chi_{j,I}, B^{(j)}_{t}\right\rangle\right |&=&\left| \sum_{\vec n} \sum_{\vec m} \chi_{j,I}(\vec n - \vec m)B^{(j)}_{s}(\vec m)\bar B^{(j)}_{t}(\vec n)\right |\\
&=&\left |\sum_{\vec n} \bar B^{(j)}_{t}(\vec n)\sum_{\vec m_I} \chi_{j,I}(n_I-m_I,0)B^{(j)}_{s}(\vec m_I, \vec n_I^\perp)\right |\\
&\leq& \sum_{\vec n} |\bar B^{(j)}_{t}(\vec n)|\cdot \|\chi_I\|_\infty \sum_{\vec m_I} |B^{(j)}_{s}(\vec m_I,\vec n_I^\perp)|\\
&\leq& \|B^{(j)}_{t}\|_1\|\chi_{j,I}\|_\infty \left(\sup_{\vec n_I^\perp}\sum_{\vec n_I} |B^{(j)}_{s}(\vec n_I,\vec n_I^\perp)|\right).
\end{eqnarray*}
Just as the delta mass at 0 for the speckled averages made it useful to define $B^{(j)}_{s}$ so that its $\ell^\infty$ norm was suitably bounded, in the plaid case we will want to define $B^{(j)}_{s}$ so that we have bounds on the collection of $\ell^\infty(\ell^1)$ mixed norms that appear in the last line above. As in the speckled case, the bounds we can obtain from properties of the support will suffice.
\\
\\ That is, instead of (\ref{Sshort}), we let $Q_{j,\vec n}$ denote the dyadic cube with side length $2^j$ containing $\vec n$, and define for $1\leq \#I<d$
\begin{eqnarray}
b^{j,I}_s(\vec n)&:=&\left\{\begin{array}{ll} b_s(\vec n), & \displaystyle\sum_{\{\vec m_I:(\vec m_I,\vec n_I^\perp)\in Q_{j,\vec n}\}}|b_s(\vec m_I,\vec n_I^\perp)|>2^{(d-\alpha(d-\# I)+\epsilon)j}, \\ 0 & \text{otherwise;}\end{array}\right.
\end{eqnarray}
For $I=\emptyset$, define as before
\begin{eqnarray}
b^{j,\emptyset}_s(\vec n)&:=&\left\{\begin{array}{ll} b_s(\vec n), & |b_s(\vec n)|>2^{(d-d\alpha)j} \\ 0 & \text{otherwise;}\end{array}\right.
\end{eqnarray}
Note that this bound lacks the factor of $2^{\epsilon j}$; this exact bound is crucial, since for this term we will have no extra cancellation in $\chi_{j,I}$ for $\alpha$ small. \\

Although we would have problems with double-counting if we simply subtracted these pieces from the original $b_s$, it is clear that we can define $B^{(j)}_s$ such that
\begin{eqnarray}
\sum_{\vec m_I:(\vec m_I,\vec n_I^\perp)\in Q_{j,\vec n}}|B^{(j)}_s(\vec m_I,\vec n_I^\perp)|\leq2^{(d-\alpha(d-\# I)+\epsilon)j}
\end{eqnarray}
and
\begin{eqnarray}
|B^{(j)}_s -b_s|\leq \sum_{I\subsetneq\{1,\dots,d\}}|b^{j,I}_s|.
\end{eqnarray}
Thus we have split up $b_s$ with regard to the size of the relevant mixed norm on the cubes $Q_{j,\vec n}$. The contribution of the terms $b^{j,\emptyset}_s$ is bounded as before. For the others, we will use the geometry of the supports and the fact that we can spare an extra $2^{\epsilon j}$.
\\
\\If we let $V_{j,I}:=\{(\vec v_I,0):|\vec v_I|<2^{j+1}\}$, then we note that $\supp \mu_j$ can be covered by $\leq C_\omega 2^{(d-\# I)(1-\alpha)j}$ copies of $V_{j,I}$, that $\supp b^{j,I}_s$ can be covered by $\leq 2^{-(d-\alpha(d-\# I)+\epsilon)j}\|b_s\|_1$ copies of $V_{j,I}$, and that $|V_{j,I}+V_{j,I}|\leq 2^{\# I(j+2)}$. Therefore
\begin{eqnarray*}
\#\{\vec n: \sup_j |\sum_s b^{j,I}_s\ast\mu_j(\vec n)|>0\}&\leq&\sum_j\sum_s |(\supp b^{j,I}_s)+(\supp \mu_j)|\\
&\lesssim& \sum_j\sum_s 2^{\# I(j+2)}2^{-(d-\alpha(d-\# I)+\epsilon)j}\|b_s\|_12^{(d-\# I)(1-\alpha)j}\\
&\leq& \|b\|_1\sum_j 2^{-\epsilon j}\lesssim \|f\|_1.
\end{eqnarray*}
Therefore we may reduce as before to considering the contribution of the $B^{(j)}_s$. The argument is again identical to the proof of Theorem \ref{SgtPepper} until the point where the inner product is calculated. For $I\neq\emptyset$ we have (again under the assumption, easily removed, that $B^{(j)}_s$ and $B^{(j)}_t$ are each supported on single dyadic cubes of size $2^j$) that
\begin{eqnarray*}
\left |\left\langle B^{(j)}_{s}\ast\nu_j\ast\tilde\nu_j, B^{(j)}_{t}\right\rangle\right |&\leq&\sum_{I\subset \{1,\dots,d\}} \left|\left\langle B^{(j)}_{s}\ast\chi_{j,I}, B^{(j)}_{t}\right\rangle\right |\\
&\leq& \|B^{(j)}_{t}\|_1\|\chi_{j,I}\|_\infty \left(\sup_{\vec n_I^\perp}\sum_{\vec n_I} |B^{(j)}_{s}(\vec n_I,\vec n_I^\perp)|\right)\\
&\lesssim& \|B^{(j)}_t\|_1 2^{(-d-\frac{\# I}2+d\alpha+\epsilon)j}2^{(d-\alpha(d-\# I)+\epsilon)j}\\
&\leq& \|B^{(j)}_t\|_1 2^{((\frac12-\alpha)\# I+2\epsilon)j}
\end{eqnarray*}
and so the terms are bounded for $\epsilon$ sufficiently small, so long as $\alpha<\frac12$.
\end{proof}

\section{\bf Sparse Arithmetic Sets in $\Z^d$}
\label{detsparse}
 In the following two sections, we seek to adapt a recent construction by M. Christ \cite{ChristPreprint} to a higher-dimensional setting. That construction takes polynomial subsets of finite groups $\Z_p^m$ and transfers them to $\Z^d$ via Freiman isomorphisms; Weil's inequality on complete character sums implies that the uniform probability measure on such a set will be optimally pseudo-random (in the sense of its Fourier transform), which suffices to prove a weak (1,1) bound on these averages.
\\
\\ As in Section \ref{random}, we have our choice between extending this one-dimensional result into a product set (plaid) construction or a natively $d$-dimensional (speckled) version. (One of the means by which we can vary the sparse sequence's rate of growth, in fact, is to project a $m$-dimensional set down to $\Z^d$, for $m$ a multiple of $d$.)  The proofs will follow closely those in \cite{ChristPreprint}; in the speckled case, we will take an isomorphism from $\Z_p^m$ to $\Z^d$ rather than $\Z$, while in the plaid case we observe that the relevant Fourier estimates factorize into a product of the estimates for the original sequence.
\\
\\ In addition, we prove  an oscillational inequality for these averages, which implies a pointwise ergodic theorem; the corresponding result was not included in \cite{ChristPreprint}.

\subsection{A Speckled Construction}
\label{detspeckled}

Let $c$ and $C$ be constants larger than $1$ and suppose that $\{ p_k \}$ is a sequence of primes satisfying
\begin{equation*}
 cp_k \leq p_{k+1} < Cp_k
\end{equation*}
for all $k$. (Later, we will want $p_k\approx 2^{\gamma k}$ for some $\gamma>0$, but this can be achieved just be repeating terms as necessary.)
\\
\\As mentioned before, we will construct $m$-dimensional sparse sets, where $m=qd$ for some positive integer $q$; the case $q=1$ gives us a ``native'' $d$-dimensional set, while higher values of $q$ give us sparser sets. So we fix a positive integer $q$ and choose a sequence of vectors $\{ \vec {a}_k \} \subset \mathbb{Z}^d$ such that
\begin{eqnarray*}
 |\vec a_{k-1}|+p_{k-1}^q < |\vec a_k| \leq Cp_{k}^{q}.
\end{eqnarray*}

Let $[n]$ denote $n\mod p_k$, and define the sequence
\begin{eqnarray}
x_{k,j} = \left(\sum_{i=1}^{q} p^{i-1} [j^i]_{p_k}, \sum_{i=1}^{q} p^{i-1} [j^{q+i}]_{p_k}, ..., \sum_{i=1}^{q} p^{i-1} [j^{(d-1)q+i}]_{p_k}  \right), 0 \leq j < p_k
\end{eqnarray}
and the set
\begin{equation}
 S_k = \left\{ \vec {a}_k + x_{k,j}: 0 \leq j < p_k\right\}.
\end{equation}
This is the image of the set $\{(j,j^2,\dots,j^m):0\leq j <p_k\}\subset \Z_{p_k}^m$ under a suitable Freiman isomorphism. The main reason we have chosen such a set is that it enjoys near-optimal Fourier bounds thanks to Weil's theorem on charcter sums:
\begin{thm}(Weil, \cite{MR0027151} ) \label{WeilsThm}
 Suppose $f \in \mathbb{Z}_p[x]$ is a polynomial and $p$ does not divide the degree, $m$, of $f$.
 Then
 \begin{equation*}
  \left| \sum_{n \in \mathbb{Z}_p}  e\left( f(n)/p \right) \right| \leq (m-1) p^{1/2}.
 \end{equation*}
\end{thm}
The size requirements on $\{ \vec {a}_k \}$ ensure that the collection $\left\{ S_k \right\}$ is disjoint; in fact, all the elements of $S_k$ lie in a shell of elements whose lengths are larger than $(C+1)p_{k}^q$ but less than $Cp_{k+1}^q$.  We also have that $\# S_k= p_k$.
\\
\\Let $S = \cup_{k=1}^\infty S_k= \{x_{k,j}:k\in\N, 0\leq j<p_k\}$ .  As a consequence of the sparseness of the individual sets $S_k$ and the lacunary nature of the sequence $\{ p_k \}$, the sparsity of the set $S$ is similar to the sequence $(n^q,\cdots,n^q)$. We wish to prove the following $L^1$ pointwise ergodic theorem:

\begin{thm}
Let $\{\vec n_i\}_{i\in\N}$ be an enumeration of $\{a_k + x_{k,j}\}$ in the dictionary ordering. Then for any measure-preserving $\Z^d$-action $\cal T$ and any $f\in L^1(X)$, the averages
\begin{eqnarray*}
A_Nf(x):=\frac1N \sum_{i=1}^N f\circ {\cal T}(\vec n_i)(x)
\end{eqnarray*}
converge almost everywhere in $X$.
\end{thm}

As usual, this result will follow from an $L^2$ pointwise result and a weak (1,1) maximal inequality. Using the standard transference arguments, it suffices to prove an oscillational inequality and a weak maximal inequality for the corresponding convolution operators on $\Z^d$.
\\
\\For the $L^2$ result, we will consider the averages on $\Z^d$,
\begin{eqnarray}
A_Nf(\vec v):= \frac1N \sum_{i=1}^N f(v+n_i).
\end{eqnarray}
We must prove the following oscillational inequality, for any given lacunary $I\subset\N$:
\begin{thm}
\label{OscIneq}
For any sequence $t_1\leq t_2\leq\cdots$ with $t_n \in I$ for all $n$,
\begin{eqnarray}
\label{Osci}
\sum_n\left\| \sup_{t_{n-1}\leq t\leq t_{n}, t\in I} |A_t f-A_{t_n}f|\right\|_2^2\leq C\|f\|_2^2.
\end{eqnarray}
\end{thm}
We will prove this (in Section \ref{osciproof}) by comparing $A_t$ to a more standard average, such that the Fourier norm of the difference is small. When transferred back to the dynamical system, such an oscillational inequality will directly imply an $L^2$ pointwise ergodic theorem (see for instance Jones, Kaufman, Rosenblatt and Wierdl \cite{JKRW}).
\\
\\Then, since the blocks grow exponentially in size, we need only prove the $L^1$ weak maximal inequality for complete blocks:

\begin{thm}\label{native}
 Let $S_k$ be defined as above.  Then the maximal function
\begin{equation*}
 Mf(\vec {n}) = \sup_{N} \frac{1}{\#\left (\cup_{k=1}^N S_k \right )} \sum_{\vec {m} \in \cup_{k=1}^N S_k} \left| f\left(\vec {n} + \vec {m} \right) \right|
\end{equation*}
satisfies a weak-(1,1) inequality; that is, there is a constant $C$ so that for any $\lambda > 0$ we have
\begin{equation*}
 \# \left\{ \vec {n} : Mf(\vec {n}) > \lambda \right\} < \frac{C}{\lambda} \left\| f \right\|_{\ell^1\left(\mathbb{Z}^d\right)}.
\end{equation*}
\end{thm}

\subsubsection{Proof of Theorem \ref{native}}\label{proofweak}

Again using a transference argument to reduce to a question about convolution operators on $\Z^d$, we seek to apply the following theorem from \cite{ChristPreprint}.

\begin{thm} \label{differencethm}
 Let $G$ be a discrete group and $\gamma > 0$.  Suppose the sequences of functions $\mu_k, \nu_k : G \rightarrow \mathbb{C}$ satisfy the following requirements:
 \begin{enumerate}
  \item the maximal operator $\sup_k |f| \ast | \nu_k |$ is of weak type (1,1) on $G$,\\
  \item for each $\mu_k$, $\#\left ( \supp (\mu_k) \right ) \leq C2^{k\gamma}$ for some constant $C$, and\\
  \item $\left\| f \ast (\mu_k - \nu_k) \right\|_{\ell^2(G)} \leq C2^{-k\gamma/2} \left\| f \right\|_{\ell^2(G)}$ for all $f \in \ell^2(G)$.\\
 \end{enumerate}
Then the maximal operator $\sup_k |f \ast \mu_k|$ is of weak type (1,1).
\end{thm}

As elsewhere, we let $e(\theta)$ denote $e^{2\pi i \theta}$. Henceforward we will suppress our $k$ subscripts and identify $\Z_p^m$ with its own dual group.
\\
\\If $f$ is a function on $\mathbb{Z}_{p}^m$, then the Fourier transform of $f$ is
\begin{equation*}
 \hat{f}(\vec {\xi}) = \sum_{\vec {n} \in \mathbb{Z}_{p}^m} f(\vec {n}) e\left( \frac{\vec {n} \cdot \vec {\xi}}{p}\right).
\end{equation*}
We then have the inequalities
\begin{align}
 \left\| \widehat{fg} \right\|_{\ell^\infty} &\leq \frac{1}{p^m} \left\| \hat{f} \right\|_{\ell^1} \left\| \hat{g} \right\|_{\ell^\infty} \mbox{, and} \label{MainF} \\
 \left\| f \ast g \right\|_{\ell^2} &\leq \frac{1}{p^{m/2}} \left\| \hat{f} \right\|_{\ell^2} \left\| \hat{g} \right\|_{\ell^\infty} = \left\| f \right\|_{\ell^2} \left\| \hat{g} \right\|_{\ell^\infty}. \label{FirstF}
\end{align}

Finally, we will again make use of Weil's theorem for complete character sums (Theorem \ref{WeilsThm}). Let $m=dq$.  Our goal will be to produce measures $\mu'$ and $\nu'$ on $\mathbb{Z}_{p}^{m}$ that satisfy the analogues of the conditions in Theorem \ref{differencethm}.  We will then apply a linear operator to move from $\mathbb{Z}_{p}^{m}$ to $\mathbb{Z}^{m}$.  A second operator will then transfer these measures to $\mathbb{Z}^{d}$.

\begin{proof}[Proof of Theorem \ref{native}]

We begin by defining a probability measure on $\mathbb{Z}_{p}^m$:
\begin{eqnarray}\label{finfieldmu}
 \mu' = \frac{1}{p} \sum_{j=0}^{p -1} \delta_{(j, j^2, ..., j^m)}.
\end{eqnarray}
Note that
\begin{eqnarray} \label{req1}
 \#\left( \supp (\mu') \right )= p.
\end{eqnarray}
Thus $\mu'$ satisfies the second requirement of Theorem \ref{differencethm} if $p_k\approx 2^{k\gamma}$.

We then define a second measure:
\begin{eqnarray}\label{finfieldnu}
 \nu' = \frac{1}{p^m} \sum_{ \vec {j} \in \mathbb{Z}_{p}^m} \delta_{\vec {j}}.
\end{eqnarray}
This will lead to a measure $\nu$ which satisfies the first requirement.
\\
\\We are left with the third condition.  By our inequality (\ref{FirstF}), we have that
\begin{equation*}
\left\| f \ast (\mu'-\nu') \right\|_{\ell^2} \leq \left\| f \right\|_{\ell^2} \left\| \widehat{\mu'-\nu'} \right\|_{\ell^\infty}.
\end{equation*}
It remains to prove an appropriate bound on $\left| \widehat{\mu'-\nu'} \right|$, using Theorem \ref{WeilsThm}. The Fourier transform of $\mu'$ is
\begin{align*}
\widehat{\mu'}(\vec {\theta}) &= \sum_{\vec {n} \in \mathbb{Z}_{p}^m} \mu'(\vec {n}) e\left( \frac{\vec {n} \cdot \vec {\theta}}{p}\right)\\
                        &= \frac{1}{p} \sum_{j=0}^{p -1} \sum_{\vec {n} \in \mathbb{Z}_{p}^m}  \delta_{(j, j^2, ..., j^m)}(\vec {n}) e\left( \frac{\vec {n} \cdot \vec {\theta}}{p}\right) \\
                        &= \frac{1}{p} \sum_{j=0}^{p -1} e\left( \frac{j\theta_1 +  j^2\theta_2 +... + j^m \theta_m}{p}\right).
\end{align*}

By the theorem, we then must have (for all $\vec\theta\neq0$)
\begin{equation*}
\widehat{\mu'}(\vec {\theta}) \leq \frac{m-1}{p^{1/2}}.
\end{equation*}

And because
\begin{align*}
 \widehat{\nu'}({\vec {\theta}}) &=  \sum_{\vec {n} \in \mathbb{Z}_{p}^m} \frac{1}{p^m} \sum_{\vec {j} \in \mathbb{Z}_{p}^m} \delta_{\vec {j}}({\bf{n}}) e\left(\frac{{\vec {n}} \cdot {\vec {\theta}}}{p}\right)  \\
 & = \left\{ \begin{array}{lr}
              1 \,\,\,&\mbox{ if } {\vec {\theta}}= {\vec {0}}\\
              0 \,\,\,&\mbox{ otherwise,}
             \end{array}
             \right.
\end{align*}

and
\begin{equation*}
 \widehat{\mu'}(\vec {0}) = 1,
\end{equation*}

we have, for $\vec {\theta} \neq \vec {0}$ that
\begin{equation}\label{Weil}
 \left| \widehat{\mu' - \nu'} \right| \leq \frac{m-1}{p^{1/2}},
\end{equation}
while for $\vec {\theta} = \vec {0}$, the difference is $0$.
\\
\\Having found suitable measures $\mu'$ and $\nu'$ on $\Z_p^m$, we now construct measures on $\Z^d$ with the same Fourier properties. We will first create a ``smooth cutoff'' version on $\Z_{3p}^m$, then transfer this to $\Z^d$; since the result majorizes our desired measures $\mu$ and $\nu$, we may thus obtain the weak inequality for them.
\\
\\First we will build a smoothing function, $\phi$.  We identify  $\mathbb{Z}_{3p}^m$ with $[-p, 2p-1]^m$ and $\mathbb{Z}_{3p}$ with $[-p, 2p-1]$.  We define the function $\varphi : \mathbb{Z}_{3p} \rightarrow \mathbb{R}$ so that
\begin{equation*}
 \varphi(n) = \left\{
		    \begin{array} {ll}
                    1  &\mbox{ if } n \in [0, p-1],\\
                    0  &\mbox{ if } n \in \left[-p, \frac{-p-1}{2} \right] \cup \left[\frac{3}{2}(p-1), 2p - 1 \right] \mbox{, and}\\
                    \mbox{affine} &\mbox{ otherwise.}
                   \end{array}
                    \right.
\end{equation*}
We then define $\phi: \mathbb{Z}_{3p}^m \rightarrow \mathbb{R}$ as a product of $\varphi$'s:
\begin{equation*}
 \phi(\vec {n}) = \prod_{i=1}^m \varphi(n_i).
\end{equation*}

We will also require the functions $\tau_p: \mathbb{Z}^m \rightarrow \mathbb{Z}_{p}^m$ and $\tau_{3p}: \mathbb{Z}^m \rightarrow \mathbb{Z}_{3p}^m$, given by
\begin{align*}
\tau_p (\vec {x}) &= \left([x_1]_p, [x_2]_p, ..., [x_m]_p \right) \mbox{, and}\\
\tau_{3p} (\vec {x}) &= \left([x_1]_{3p}, [x_2]_{3p}, ..., [x_m]_{3p} \right)
\end{align*}

We can now transfer from measures on $\mathbb{Z}_{p}^m$ to ``smoothly cut-off'' measures on $\mathbb{Z}^m$. Define $\Gamma_1: \ell^\infty (\mathbb{Z}_{p}^m) \rightarrow  \ell^\infty (\mathbb{Z}^m)$ by
\begin{equation*}
 \Gamma_1(f)(\vec j) = \ind_{[-p, 2p-1]^m}(\vec j) \phi(\tau_{3p}(\vec j)) f(\tau_p(\vec j)) .
\end{equation*}

Let $\vec {\theta} \in \mathbb{T}^m$ and write $\vec {\theta} = \vec {\xi}/3p + \vec {\eta}$ with $\vec {\xi} \in \mathbb{Z}^m$ and $\left | \eta_i \right| \leq C/p$ for all $1 \leq i \leq m$. The Fourier transform of $\Gamma_1(f)$ is
\begin{align*}
 \widehat{\Gamma_1 (f)}(\vec {\theta}) &= \sum_{\vec {j} \in \mathbb{Z}^m} \Gamma_1 (f) (\vec {j}) e\left( \vec {j} \cdot \vec {\theta} \right)\\
				   &= \sum_{\vec {j} \in [-p, 2p-1]^m} \phi(\vec {j}) f(\tau_p(\vec {j})) e\left( \vec {j} \cdot \vec {\eta} \right) e\left( \frac{\vec {j} \cdot \vec {\xi}}{3p} \right).
\end{align*}

Changing perspective, we may consider this as the Fourier transform on the group $\Z_{3p}^m$ of the product of $f$ and $\psi(\vec {x})= \phi(\vec x) e\left(\vec x \cdot \vec {\eta} \right)$. By the inequality (\ref{MainF}), then, we have
\begin{equation*}
 \left\| \widehat{\Gamma_1 (f)}(\vec {\theta}) \right\|_{L^\infty(\T^m)} \leq \frac{1}{p^m} \left\| \hat f \right\|_{\ell^\infty(\Z_{3p}^m)} \left\| \hat{\psi} \right\|_{\ell^1(\Z_{3p}^m)}.
\end{equation*}

We plan to set
\begin{align*}
 \mu'' &= \Gamma_1( \mu') \mbox{ and}\\
 \nu'' &= \Gamma_1( \nu').
\end{align*}
The inequalities (\ref{req1}) and (\ref{Weil}) will insure that the requirements of Theorem \ref{differencethm} are met for $\mu''$ and $\nu''$ on $\Z^m$, so long as $\|\hat\psi\|_1\leq Cp^m$. Now
\begin{align*}
\left\| \hat{\psi} \right\|_{1} &= \sum_{\vec {\xi} \in \mathbb{Z}_{3p}^m } \left| \sum_{\vec {j} \in \Z_p^m} \phi(\vec {j}) e\left(\vec {j} \cdot \vec {\eta} \right) e(\frac{\vec {j} \cdot \vec {\xi}}{3p}) \right| \\
& = \sum_{\vec {\xi} \in \mathbb{Z}_{3p}^m } \left| \sum_{\vec {j} \in [-p, 2p-1]^m} \prod_{i=1}^m \varphi(j_i) e(j_i \eta_i) e\left( \frac{j_i \xi_i}{3p} \right) \right| \\
& = \prod_{i=1}^m \left( \sum_{\xi_i \in \mathbb{Z}_{3p}} \left| \sum_{j_i \in [-p, 2p-1]}  \varphi(j_i) e(j_i \eta_i) e\left( \frac{j_i \xi_i}{3p} \right) \right| \right).
\end{align*}

It therefore suffices to show
\begin{equation}\label{cpest}
  \sum_{\xi \in \mathbb{Z}_{3p}} \left| \sum_{j \in [-p, 2p-1]}  \varphi(j) e(j \eta) e\left( \frac{j \xi}{3p} \right) \right|  \leq Cp.
\end{equation}

For $\xi = 0$, we have the trivial bound
\begin{equation*}
 \left| \sum_{j \in [-p, 2p-1]}  \varphi(j) e(j \eta) \right| \leq  \sum_{j \in [-p, 2p-1]} \left| \varphi(j) \right| \leq 3p.
\end{equation*}

For $\xi \neq 0$, we will pursue the required bound using summation by parts. Letting $\Phi(j) = \varphi(j)e(j \eta)$, extended periodically, and noting that $\Delta \Phi (n) = \Phi (n+1) - \Phi(n)$, we have
\begin{align*}
 \sum_{j \in [-p, 2p-1]}  \Phi(j) e\left( \frac{j \xi}{3p} \right)&= \Phi(2p) \left( \sum_{j=-p}^{2p-1}  e\left( \frac{j \xi}{3p} \right)  \right) - \sum_{j=-p}^{2p-1} \Delta \Phi(j)  \sum_{n=-p}^j e\left( \frac{n \xi}{3p} \right)\\
&=  - \sum_{j=-p}^{2p-1} \Delta \Phi(j)  \frac{e\left( \frac{(j+1) \xi}{3p} \right) - e\left( \frac{\xi}{3}\right)}{e\left(\frac{\xi}{3p}\right) - 1},
\end{align*}
since $\Phi(2p)=0$

Further, we have
\begin{align*}
 - \sum_{j=-p}^{2p-1} &\Delta \Phi(j)  \frac{e\left( \frac{(j+1) \xi}{3p} \right) - e\left( \frac{\xi}{3}\right)}{e\left(\frac{\xi}{3p}\right) - 1}\\
&= -\left( e\left(\frac{\xi}{3p}\right) - 1 \right)^{-1} \left( \sum_{j=-p}^{2p-1} \Delta \Phi(j) e\left( \frac{(j+1) \xi}{3p} \right) - \sum_{j=-p}^{2p-1} \Delta \Phi(j) e\left( \frac{\xi}{3}\right) \right)\\
&= -\left( e\left(\frac{\xi}{3p} \right)- 1 \right)^{-1} \left( \sum_{j=-p}^{2p-1} \Delta \Phi(j) e\left( \frac{(j+1) \xi}{3p} \right) - e\left( \frac{\xi}{3}\right) \sum_{j=-p}^{2p-1} \Phi(j+1) - \Phi(j) \right)\\
&=  -\left( e\left(\frac{\xi}{3p} \right)- 1 \right)^{-1}  e\left( \frac{\xi}{3p} \right)\sum_{j=-p}^{2p-1} \Delta \Phi(j) e\left( \frac{(j) \xi}{3p} \right).
\end{align*}

We now apply summation by parts a second time:
\begin{align*}
 \sum_{j=-p}^{2p-1} &\Delta \Phi(j) e\left( \frac{(j) \xi}{3p} \right)\\
&= \Delta \Phi (2p) \left( \sum_{j=-p}^{2p-1}  e\left( \frac{(j) \xi}{3p} \right)  \right) - \sum_{j=-p}^{2p-1} \Delta^2 \Phi(j)  \sum_{n=-p}^j e\left( \frac{(n) \xi}{3p} \right)\\
&= -\sum_{j=-p}^{2p-1} \Delta^2 \Phi(j) \frac{e\left( \frac{(j) \xi}{3p} \right) - e\left( \frac{\xi}{3}\right)}{e\left(\frac{\xi}{3p}\right) - 1}\\
&= -\left( e\left(\frac{\xi}{3p} \right)- 1 \right)^{-1} \sum_{j=-p}^{2p-1} \Delta^2 \Phi(j) e\left( \frac{(j+1) \xi}{3p} \right)
\end{align*}

We therefore have that
\begin{align*}
 \left|  \sum_{j \in [-p, 2p-1]}  \varphi(j) e(j \eta) e\left( \frac{j \xi}{3p} \right) \right| &= \left| \left( e\left(\frac{\xi}{3p} \right)- 1 \right)^{-2} \sum_{j=-p}^{2p-1} \Delta^2 \Phi(j) e\left( \frac{(j+2) \xi}{3p} \right) \right| \\
&\leq \left| \left( e\left(\frac{\xi}{3p} \right)- 1 \right)^{-2} \right| \sum_{j=-p}^{2p-1} \left| \Delta^2 \Phi(j)\right|.
\end{align*}

As a consequence,
\begin{equation*}
   \sum_{\substack{ \xi \in \mathbb{Z}_{3p} \\ \xi \neq 0}} \left| \sum_{j \in [-p, 2p-1]}  \varphi(j) e(j \eta) e\left( \frac{j \xi}{3p} \right) \right|  \leq \left\|\Delta^2 \Phi(j) \right\|_{\ell^1} \sum_{\substack{ \xi \in \mathbb{Z}_{3p} \\ \xi \neq 0}} \left| \left( e\left(\frac{\xi}{3p} \right)- 1 \right)^{-2} \right|  .
\end{equation*}

We will estimate each factor separately, making use of the inequality $|e(\theta) - 1| < |2\pi \theta|$; a consequence of the fact that a chord is necessarily shorter than the arc it subtends. For the second factor, this gives us
\begin{equation*}
\sum_{\substack{ \xi \in \mathbb{Z}_{3p} \\ \xi \neq 0}} \left| \left( e\left(\frac{\xi}{3p} \right)- 1 \right)^{-2} \right|  < 9p^2 \sum_{\xi = 1}^\infty \frac{1}{(2 \pi \xi)^2} = Cp^2.
\end{equation*}

For the second, we have
\begin{align}
 \sum_{j=-p}^{2p-1} \left| \Delta^2 \Phi(j)\right| &= \sum_{j=-p}^{2p-1} \left| \Phi(j+2) - 2 \Phi(j+1) + \Phi(j) \right| \notag\\
    &= \sum_{j=-p}^{2p-1} \left| \varphi(j+2)e((j+2)\eta) - 2\varphi(j+1)e((j+1)\eta)  + \varphi(j)e(j\eta) \right| \notag\\
    &= \sum_{j=-p}^{2p-1} \left| \varphi(j+2)e(2\eta) - 2\varphi(j+1)e(\eta)  + \varphi(j)\right| .
\end{align}

For $j \in [-p, \frac{-p-1}{2}-3] \cup [\frac{3}{2}(p-1), 2p-1]$ this last sum is $0$.  For $j \in [0, p-3]$, we have $\varphi(j+2) = \varphi(j+1) = \varphi(j) = 1$, hence
\begin{align*}
 \left| \varphi(j+2)e(2\eta) - 2\varphi(j+1)e(\eta)  + \varphi(j)\right| &= \left|e(2\eta) - 2e(\eta) + 1 \right| \\
  &=|e(\eta) - 1|^2 < (2\pi \eta)^2
\end{align*}
Since $\eta < 1/p$, this is less than $40/p^2$.  The sum over this subinterval is therefore less than $\frac{40(p-3)}{p^2} < 40/p$.

For $j \in [\frac{-p-1}{2} - 1, -1]$, we have
\begin{align*}
 \left| \varphi(j+2)e(2\eta) - 2\varphi(j+1)e(\eta)  + \varphi(j)\right| &= \left| (\varphi(j) + \frac{4}{p-1})e(2\eta) - 2(\varphi(j) + \frac{2}{p-1})e(\eta)  + \varphi(j)\right|\\
&< \left|\varphi(j)\right| \frac{40}{p^2} + \left|\frac{4}{p-1}e(2\eta) - \frac{4}{p-1}e(\eta) \right|\\
&= \left|\varphi(j)\right| \frac{40}{p^2} + \frac{4}{p-1} \left|e(\eta) - 1 \right| \\
&< \frac{40+8\pi}{p^2}.
\end{align*}

A similar calculation gives the same bound over the interval $[p-1, \frac{3}{2}(p-1)-2]$.  Thus we also have a bound of $C/p$ on these sums.

This leaves us with only the cases $j = \frac{-p-1}{2} - 2$ and $j=\frac{3}{2}(p-1)-1$.  For, $j=\frac{-p-1}{2} - 2$, though, we have that $\varphi(j) = \varphi(j+1) = 0$; this leaves us with
\begin{equation*}
 \left| \varphi(\frac{-p-1}{2}) e(2\eta) \right| = \frac{2}{p-1}.
\end{equation*}
The same bound holds for the case $j= \frac{3}{2}(p-1)-1$.

We therefore have that
\begin{equation*}
  \sum_{j=-p}^{2p-1} \left| \Delta^2 \Phi(j)\right| < \frac{C}{p}.
\end{equation*}

This completes our estimate (\ref{cpest}).

\begin{remark}
In the case $q=1$ (recall that $m=dq$), we are actually finished, since our desired measure $\mu$ is majorized by $\mu''$. We wish to generalize, however, to show that we can achieve any desired polynomial rate of sparsity. Therefore, we shall introduce a second operator to ``project down'' from $\Z^m$ to $\Z^d$.
\end{remark}

Define the Freiman isomorphism $F: \mathbb{Z}^m \rightarrow \mathbb{Z}^d$ by
\begin{equation}\label{Freiman}
 F(\vec {j}) = \left( \sum_{i=1}^q p^{i-1} j_i, \sum_{i=q+1}^{2q} p^{i-(q+1)} j_i, ..., \sum_{i=q(d-1)+1}^{dq} p^{i-(q(d-1)+1)} j_i \right).
\end{equation}

We note that $F$ maps $[0, p-1]^m$ bijectively to $[0, p^q-1]^d$, and that for any $k \geq 0$ and any $j_i \in [-p, 2p-1]$ we have $\sum_{i=kq+1}^{(k+1)q} p^{i-1} j_i \in [-qp^q, qp^q]$ ; hence, for $\vec {j} \in [-p, 2p-1]^m$, we have
\begin{equation*}
F(\vec {j}) \in \left[ -qp^q, qp^q \right]^d.
\end{equation*}

Define the operator $\Gamma_2$ by
\begin{equation*}
 \Gamma_2 f = \sum_{\left\{\vec {j}: F(\vec {j}) = \vec {n} \right\}} f(\vec {j}).
\end{equation*}

Now consider the operator $\Gamma=\Gamma_2 \Gamma_1:\ell^1(\Z_p^m)\to\ell^1(\Z^d)$.  As before, we define our intermediate measures $\mu'''$ and $\nu'''$,

\begin{align*}
\mu''' &= \Gamma_2(\mu'') \mbox{, and}\\
\nu''' &= \Gamma_2(\nu''),
\end{align*}
and we want to ensure that $\|\widehat{\Gamma f}\|_{L^\infty(\T^d)}\leq C\|\hat f\|_{\ell^\infty(\Z_p^m)}$ so that we may apply Theorem \ref{differencethm}.
\\
\\Now,
\begin{align*}
\widehat{\Gamma f}&(\vec {\theta})  = \sum_{\vec {n} \in \mathbb{Z}^d} \Gamma_2 \Gamma_1 f (\vec {n}) e(\vec {\theta} \cdot \vec {n}) \\
    &= \sum_{\vec {n} \in \mathbb{Z}^d} e(\vec {\theta} \cdot \vec {n}) \sum_{\left\{\vec {j}: F(\vec {j}) = \vec {n} \right\}} \Gamma_1 f(\vec {j}) \\
    &= \sum_{\vec {j} \in [-p, 2p-1]^m} \Gamma_1 f (\vec {j}) e\left( \theta_1 \sum_{i=1}^q p^{i-1} j_i  + \theta_2 \sum_{i=q+1}^{2q} p^{i-(q+1)} j_i + ...+ \theta_d  \sum_{i=q(d-1)+1}^{dq} p^{i-(q(d-1)+1)} j_i \right)\\
    &= \widehat{\Gamma_1 f}\left(\theta_1, \theta_1p, ..., \theta_1p^{q-1}, \theta_2, ..., \theta_q, \theta_q p, ..., \theta_d p^{q-1} \right).
\end{align*}

So, $\left\|\widehat{\Gamma_2 \Gamma_1 f} \right\|_{\ell^\infty} \leq \left\|\widehat{\Gamma_1 f} \right\|_{\ell^\infty}$, and by our previous bound we have that $\mu''' = \Gamma_2 \Gamma_1 \mu'$  and $\nu''' = \Gamma_2 \Gamma_1 \nu'$ obey the difference requirement from Theorem \ref{differencethm}.  Further, we have that
\begin{equation} \label{muineq}
\mu'''(\vec {n}) \geq \frac{1}{p^m} \sum_{\vec {j} \in \bar{S}} \delta_{\vec {j}}(\vec {n}),
\end{equation}
where
\begin{equation*}
\bar{S} = \left\{  \left(\sum_{i=1}^{q} p^{i-1} [j^i]_{p}, \sum_{i=1}^{q} p^{i-1} [j^{q+i}]_{p}, ..., \sum_{i=1}^{q} p^{i-1} [j^{(d-1)q+i}]_{p}  \right): 0 \leq j < p \right\},
\end{equation*}
so that a weak (1,1) inequality for $\mu'''$ implies a weak (1,1) inequality for $\mu$. We now wish to show that $\mu'''$ and $\nu'''$ satisfy the other two requirements of Theorem \ref{differencethm}.
\\
\\We first observe that
\begin{equation*}
\#\left ( \supp\left( \Gamma_1 \mu' \right) \right ) = \#\left ( \supp \left( \ind_{[-p, 2p-1]^m} \phi \circ \tau_{3p} \mu' \circ \tau_p \right) \right ) \leq Cp.
\end{equation*}

Recalling our requirements on our original sequence of primes, and noting that $\#\left ( \supp \left(\Gamma_2f\right) \right ) \leq \#\left ( \supp (f) \right )$, our condition on the support of $\mu'''$ is satisfied.

For the final requirement, we first note that $\#\left ( \supp \Gamma_1 \nu' \right ) \leq (3p)^m$; the support of $\nu'''$, then, likewise has measure less than $Cp^m$.  The supremum of $\nu'$ is $1/p^m$.  Thus we have
\begin{align*}
\left\| \nu''' \right\|_{\ell^1} &\leq Cp^m \left\| \Gamma_2 \Gamma_1 \nu' \right\|_{\ell^\infty} \\
&\leq Cp^m \sup_{\vec {n}} \left| \sum_{\left\{\vec {j}: F(\vec {j}) = \vec {n} \right\}} \ind_{[-p, 2p-1]^m}  \frac{\phi}{p^m} \right| \\
&\leq C 3^m \sum_{\substack{\left\{\vec {j}: F(\vec {j}) = \vec {n} \right\} \\ \vec {j} \in [0,p-1]^m }} 1\\
&\leq C 3^m.
\end{align*}
We then have a weak-(1,1) bound for $\sup_k |f| \ast | \nu_k''' |$.

Restoring the subscripts, define
\begin{align*}
\mu_k(\vec {n})&= \mu_k ''' (\vec {k} - \vec {a}_k) \mbox{, and}\\
\nu_k(\vec {n})&= \nu_k ''' (\vec {k} - \vec {a}_k).
\end{align*}

By Theorem \ref{differencethm}, $\sup_k \left| f \ast \mu_k \right|$ obeys a weak-(1,1) inequality.  Together with the inequality (\ref{muineq}) this completes the proof.

\end{proof}

\subsubsection{Proof of Theorem \ref{OscIneq}}
\label{osciproof}

As in Section \ref{proofweak}, we will use the exceptionally good Fourier bounds of certain measures on $\Z_p^m$, and transfer these to measures on $\Z^d$ using operators $\Gamma_p$. However, in this case we cannot allow ourselves to use a smooth cutoff function, because we need to wind up with the actual averages, not simply weighted averages which majorize them. This introduces a logarithmic factor which would have been fatal to the weak $L^1$ maximal inequality, but which is harmless here.
\\
\\To begin, given the sequence $n_i=a_k+ x_{k,j}$ (where of course $0\leq j<p_k$) in the dictionary ordering (and thus given indices $k(i)$ and $j(i)$ for each $i\in\N$), we define the measures on $\Z^d$
\begin{eqnarray}
\mu_N &=& \frac1N \sum_{i=1}^N \delta_{n_i},\\
\nu_N &=& \frac1N\left(\frac{j(N)}{p_{k(N)}^m}\chi\left([0,p_{k(N)}^q-1]^d\right)+\sum_{k=1}^{k(N)-1}p_k^{1-m}\chi\left([0,p_k^q-1]^d\right)\right).
\end{eqnarray}
$\mu_N$ simply corresponds to an average over our sparse sequence, while $\nu_N$ is a weighted average over the $d$-dimensional blocks which our sequence ``lives on''. We will see that $\mu_N$ and $\nu_N$ are made from the images under the operators $\Gamma_p$ of the measures in (\ref{finfieldmu}) and (\ref{finfieldnu}) on $\Z_{p_k}^m$, and thus we may count on their Fourier transforms to be very close to one another.
\\
\\ Recalling the Freiman isomorphisms $F_p:\Z^m\to\Z^d$ from (\ref{Freiman}), and identifying $Z_p^m$ with $[0,p-1]^m\subset \Z^m$, define $\Gamma_p:\ell^1(\Z_p^m)\to \ell^1(\Z^d)$ by
\begin{eqnarray*}
\Gamma_p(f)(\vec n) = f(F_p^{-1}(\vec n))\chi_{[0,p^q-1]^d}(\vec n).
\end{eqnarray*}
($F_p^{-1}(\vec v)$ here denotes, for $\vec v\in [0,p^q-1]^d$, the unique $\vec w \in \Z^m$ such that $F_p(\vec w)=\vec v$.) This is the same as the final operator from the last section, except that we do not use a smooth cutoff function.
\\
\\By the same argument there, we may conclude that 
\begin{eqnarray*}
\|\widehat{\Gamma_pf}\|_{L^\infty(\T^d)} \leq Cp^{-m}\|\hat\psi\|_{\ell^1(\Z_p^m)}\|\hat f\|_{\ell^\infty(\Z_p^m)},
\end{eqnarray*}
where now $\psi(\vec n)=e(\vec n \cdot \vec \eta)$ for some $\vec\eta\in\T^m$ with $|\eta_i|<1/p$ for $i=1,\dots,m$. But trivially, $\|\hat\psi\|_{\ell^1(\Z_p^m)}\leq (Cp \log p)^m$, so $\|\widehat{\Gamma_pf}\|_{L^\infty(\T^d)} \leq C(\log p)^{m}\|\hat f\|_{\ell^\infty(\Z_p^m)}$.
\\
\\ For all of the pieces of $\mu_N$ and $\nu_N$ with $k<k(N)$, we may of course apply Weil's theorem on complete character sums (Theorem \ref{WeilsThm}). However, the last component corresponds to an incomplete character sum, for which we instead apply Weyl's Inequality \cite{MR1511670} to improve on the trivial bound; we find overall that
\begin{eqnarray*}
\|\hat\mu_N - \hat\nu_N\|_{L^\infty(\T^d)}\leq CN^{-\epsilon}
\end{eqnarray*}
for some $\epsilon>0$. Thus along any lacunary $I\subset \N$, we see that
\begin{eqnarray*}
\sum_{t\in I}\|(\mu_t-\nu_t)\ast f\|_{\ell^2(\Z^d)}^2\leq C\|f\|_{\ell^2(\Z^d)}^2.
\end{eqnarray*}
Therefore an oscillation inequality for convolution with the $\nu_t$ would imply (\ref{Osci}).
\\
\\ As in \cite{PETHA}, we introduce simple Fourier multiplier operators $V_t$ on $\ell^2(\Z^d)$, defined by
\begin{eqnarray*}
\hat V_tf(\vec \alpha)=\left\{\begin{array}{ll}\hat f(\vec \alpha), & |\alpha| \leq p_{k(t)}^{-1},\\ 0 & \text{ otherwise }\end{array}\right.
\end{eqnarray*}
Now we can pass from an oscillational inequality for the $\nu_t$ to that for the $V_t$, because
\begin{lem}
\label{FMult}
\begin{eqnarray*}
\sum_{t\in I}\|\nu_t\ast f-V_tf\|_2^2\leq C\|f\|_2^2.
\end{eqnarray*}
\end{lem}
\begin{proof}
This follows from the assertion
\begin{eqnarray*}
\sup_{\alpha \in \T^d}\left(\sum_{t\in I}|\hat \nu_t(\alpha)-\hat V_t(\alpha)|\right)\leq C.
\end{eqnarray*}
Fix $\alpha\in\T^d$, and take $K$ such that $|\alpha|\approx p_K^{-1}$. Then for $t$ with $k(t)<K$, a simple calculation shows that 
\begin{eqnarray*}
|\hat \nu_t(\alpha)-\hat V_t(\alpha)|=|\hat \nu_t(\alpha)-1|\leq |\alpha|\|\nabla \hat \nu_t\|_\infty\lesssim \displaystyle\frac{p_{k(t)}}{p_K}.
\end{eqnarray*}
Note that the calculation of $\|\nabla\hat \nu_t\|_\infty$ is the one and only place we use the growth assumption on $\vec a_k$.
\\
\\ For $t$ with $k(t)>K$, 
\begin{eqnarray*}
|\hat \nu_t(\alpha)-\hat V_t(\alpha)|=|\hat \nu_t(\alpha)|\lesssim \frac1t \prod_{i=1}^d\displaystyle\frac{1}{|\alpha_i|+1}\lesssim \displaystyle\frac{p_K}{p_{k(t)}}.
\end{eqnarray*}
 This calculation uses the fact that $\hat \nu_t(\alpha)$ can be expressed as a weighted sum of averages over blocks which grow exponentially in size. Thus 
\begin{eqnarray*}
\left(\sum_{t\in I}|\hat A_t(\alpha)-\hat V_t(\alpha)|\right)&\lesssim& \sum_{t\in I: k(t)<K} \frac{p_{k(t)}}{p_K} + O(1) + \sum_{t\in I: k(t)>K} \frac{p_K}{p_{k(t)}}\leq C
\end{eqnarray*}
since the $p_k$ are an exponentially increasing sequence and $I$ is lacunary (thus the number of $t$ associated to any $k$ is uniformly bounded).
\end{proof}

We now need only to prove that
\begin{eqnarray*}
\sum_n\left\| \sup_{t_{n-1}\leq t\leq t_{n}, t\in I} |V_t f-V_{t_n}f|\right\|_2^2\leq C\|f\|_2^2.
\end{eqnarray*}
Now we note that for $t_{n-1}\leq t\leq t_{n}$, $V_t f-V_{t_n}f=V_{t}(V_{t_{n-1}} f-V_{t_n}f)$. Also, Lemma \ref{FMult} lets us derive a $\ell^2$ maximal theorem for the $V_t$ from the one for $A_t$; this follows from our $\ell^1$ maximal inequality. Therefore

\begin{eqnarray*}
\left\| \sup_{t_{n-1}\leq t\leq t_{n}, t\in I} |V_t f-V_{t_n}f|\right\|_2^2 &=& \left\| \sup_{t_n\leq t\leq t_{n+1}, t\in I} |V_{t}(V_{t_{n-1}} f-V_{t_n}f)|\right\|_2^2\\
&\leq& C \left\|V_{t_{n-1}} f-V_{t_n}f\right\|_2^2.
\end{eqnarray*}

And now we see that
\begin{eqnarray}
\sum_{n}\|V_{t_{n-1}}f-V_{t_n}f\|_2^2\leq \left(\sup_{\alpha\in\T^d} \sum_n|\hat V_{t_{n-1}}(\alpha)-\hat V_{t_n}(\alpha)|\right)\sum_n\|f\|_2^2=\|f\|_2^2
\end{eqnarray}
since the functions $\hat V_{t_{n-1}}-\hat V_{t_n}$ have disjoint supports. This concludes the proof of Theorem \ref{OscIneq}.

\subsection{The Product Construction}
\label{detplaid}

Suppose that $m=qd$ and $S_1, S_2, ..., S_r$ are subsets of $\mathbb{Z}^{d}$, respectively, as constructed above. That is, $S_i = \cup_{k = 1}^\infty S_{i,k}$, where
\begin{equation*}
 S_{i,k} = \left\{ \vec {a}_{i,k} + \left(\sum_{t=1}^{q} p_{i,k}^{t-1} [j_{i}^t]_{p_{i,k}}, \sum_{t=1}^{q} p_{i,k}^{t-1} [j_{i}^{q+t}]_{p_{i,k}}, ..., \sum_{t=1}^{q} p_{i,k}^{t-1} [j_{i}^{(d-1)q+t}]_{p_{i,k}}  \right): 0 \leq j_i < p_{i,k} \right\},
\end{equation*}
and $1 \leq i \leq r$. The sequences $\left\{ \vec {a}_{i,k} \right\}_{k=1}^\infty$ and $\left\{ p_{i,k} \right\}_{k=1}^\infty$ are not necessarily distinct with respect to $i$.

Let $S = \prod_{i=1}^r S_i$, and let $B_t$ denote the ball of radius $t$ in $\mathbb{Z}^{rd}$.

\begin{prop}
\label{plaiddet}
Let $S$ be defined as above.  Then the maximal function
\begin{equation*}
 Mf(\vec {n}) = \sup_{t} \frac{1}{\#\left (S \cap B_t \right )} \sum_{\vec {m} \in S \cap B_t} \left| f\left(\vec {n} + \vec {m} \right) \right|
\end{equation*}
satisfies a weak-(1,1) inequality; that is, there is a constant $C$ so that for any $\lambda > 0$ we have
\begin{equation*}
 \# \left\{ \vec {n} : Mf(\vec {n}) > \lambda \right\} < \frac{C}{\lambda} \left\| f \right\|_{\ell^1\left(\mathbb{Z}^{rd}\right)}.
\end{equation*}
\end{prop}

As in the previous section, we will require two inequalities relating convolutions, products, and norms of Fourier transforms on these finite groups.
With the Fourier transform on $\prod_{i=1}^r \mathbb{Z}_{p_i}^m$ defined by
\begin{equation*}
\hat{f}(\vec {\xi}) = \sum_{ \vec{n} \in  \prod_{i=1}^r \mathbb{Z}_{p_i}^m} f(\vec {n}) \left(\frac{\vec {n}_1 \cdot \vec {\xi}_1}{p_1} +...+\frac{\vec {n}_r \cdot \vec {\xi}_r}{p_r} \right),
\end{equation*}
where $\vec {n}_i$ and $\vec {\xi}_i$ are elements of $\mathbb{Z}_{p_i}^m$, we have the natural analogues of the inequalities (\ref{MainF}) and (\ref{FirstF}):

\begin{align}
 \left\| \widehat{fg} \right\|_{\ell^\infty} &\leq \frac{1}{\prod_{i=1}^r p_{i}^m} \left\| \hat{f} \right\|_{\ell^1} \left\| \hat{g} \right\|_{\ell^\infty} \mbox{, and} \label{MainF2} \\
 \left\| f \ast g \right\|_{\ell^2} &\leq \left\| f \right\|_{\ell^2} \left\| \hat{g} \right\|_{\ell^\infty}. \label{FirstF2}
\end{align}

\begin{proof}
The proof proceeds in very much the same way as that of Theorem \ref{native}.

Once again suppressing our $k$ subscripts, let $p_1, p_2, ..., p_r$ be odd primes, each larger than $m \geq 1$. Define
\begin{equation*}
\mu_r' = \frac{1}{\prod_{i=1}^r p_i}  \sum_{\substack{(j_1, j_2, ...,j_r ) \\ \in \prod_{i=1}^r [0, p_i-1]}} \delta_{\left(j_{1},j_{1}^2, ...,j_{1}^m; j_{2},j_{2}^2, ..., j_{2}^m; ...; j_{r},j_{r}^2,...,j_{r}^m \right)},
\end{equation*}
noting that
\begin{equation*}
\#\left ( \supp(\mu_r') \right ) = \prod_{i=1}^r p_i.
\end{equation*}

We also define
\begin{equation*}
\nu_r' = \left( \prod_{i=1}^r p_i \right)^{-m} \sum_{\vec {j} \in \prod_{i=1}^r \mathbb{Z}_{p_i}^m} \delta_{\vec {j}}.
\end{equation*}

As before, we will first seek an appropriate bound on $\left| \widehat{\mu_r' - \nu_r'} \right|$.

As our products remain finite abelian groups, we have that the Fourier transform of $\mu_r'$ may be written
\begin{align*}
\hat{\mu_r'}(\vec {\theta}) &= \sum_{\vec {n} \in \prod_{i=1}^r \mathbb{Z}_{p_i}^m} \mu_r'(\vec {n}) e\left(\frac{\vec {n}_1 \cdot \vec {\theta}_1}{p_1} + \frac{\vec {n}_2 \cdot \vec {\theta}_2}{p_2} + ... +\frac{\vec {n}_r \cdot \vec {\theta}_r}{p_r}\right)\\
&=\frac{1}{\prod_{i=1}^r p_i}  \sum_{\substack{(j_1, j_2, ...,j_r )\\ \in \prod_{i=1}^r [0, p_i-1]}} \sum_{\vec {n} \in \prod_{i=1}^r \mathbb{Z}_{p_i}^m} \delta_{\left(j_{1},...,j_{1}^m; ...; j_{r},...,j_{r}^m \right)}(\vec {n}) e\left(\frac{\vec {n}_1 \cdot \vec {\theta}_1}{p_1} \right)e\left( \frac{\vec {n}_2 \cdot \vec {\theta}_2}{p_2} \right)...e\left( \frac{\vec {n}_r \cdot \vec {\theta}_r}{p_r}\right)\\
&=\frac{1}{\prod_{i=1}^r p_i} \sum_{j_1 = 0}^{p_1-1} ...\sum_{j_r = 0}^{p_r-1}  e\left(\frac{j_1\theta_1 + j_{1}^2\theta_2+...+j_{1}^m\theta_m}{p_1} \right)  ...e\left(\frac{j_r\theta_{(r-1)m+1} + j_{r}^2\theta_{(r-1)m+2}+...+j_{r}^m\theta_{rm}}{p_r} \right)\\
&= \prod_{i=1}^r \left( \frac{1}{p_i} \sum_{j_i = 0}^{p_i-1} e\left(\frac{j_i\theta_{(i-1)m+1} + j_{i}^2\theta_{(i-1)m+2}+...+j_{i}^m\theta_{im}}{p_i} \right)\right)
\end{align*}
where $\vec {n_i}$ denotes those entries of $\vec {n}$ drawn from $\mathbb{Z}_{p_i}$ and $\vec {\theta}_i =(\theta_{(i-1)m+1}, \theta_{(i-1)m+2},...,\theta_{im})$.

By Theorem \ref{WeilsThm} we then must have
\begin{equation*}
\left| \hat{\mu_r'} \right| \leq \frac{(m-1)^r}{\prod_{i=1}^r p_{i}^{1/2}}.
\end{equation*}

We have that
\begin{equation*}
\hat{\nu_r'} (\vec {\theta}) = \left\{ \begin{array}{lr}
              1 \,\,\,&\mbox{ if } {\vec {\theta}}= {\vec {0}}\\
              0 \,\,\,&\mbox{ otherwise,}
             \end{array}
             \right.
\end{equation*}
and that $\hat{\mu_r'}(\vec {0}) = 1$. Therefore,

\begin{equation}\label{ProductWeil}
 \left| \widehat{\mu_r' - \nu_r'} \right| \leq \frac{(m-1)^r}{\prod_{i=1}^r p_{i}^{1/2}}.
\end{equation}

We now embark on the construction of $\Gamma_1$ and $\Gamma_2$.

Identifying $\mathbb{Z}_{3p_i}^m$ with $[-p_i, 2p_i-1]^m$ and $\mathbb{Z}_{3p_i}$ with $[-p_i, 2p_i-1]$.  We define the functions $\varphi_i : \mathbb{Z}_{3p_i} \rightarrow \mathbb{R}$ by
\begin{equation*}
 \varphi_i(n) = \left\{
		    \begin{array} {ll}
                    1  &\mbox{ if } n \in [0, p_i-1],\\
                    0  &\mbox{ if } n \in \left[-p_i, \frac{-p_i-1}{2} \right] \cup \left[\frac{3}{2}(p_i-1), 2p_i - 1 \right] \mbox{, and}\\
                    \mbox{affine} &\mbox{ otherwise.}
                   \end{array}
                    \right.
\end{equation*}

With $\vec {n}_i=(n_{i,1}, n_{i,2}, ..., n_{i,m})$, we then define $\phi : \prod_{i=1}^r \mathbb{Z}_{3p_i}^m \rightarrow \mathbb{R}$ by
\begin{equation*}
\phi(\vec {n}) = \prod_{i=1}^r \prod_{k=1}^m \phi_i\left( n_{i,k} \right).
\end{equation*}

Define $\tau_p : \mathbb{Z}^{rm} \rightarrow \prod_{i=1}^r \mathbb{Z}_{p_i}^m$ and $\tau_{3p} :  \mathbb{Z}^{rm} \rightarrow \prod_{i=1}^r \mathbb{Z}_{3p_i}^m$ by
\begin{align*}
\tau_p (\vec {k}) &= \left([k_1]_{p_1},[k_2]_{p_1}, ..., [k_m]_{p_1}, [k_{m+1}]_{p_2}, [k_{m+2}]_{p_2}, ..., [k_{2m}]_{p_2}, ...,[k_{(r-1)m+1}]_{p_r}, ..., [k_{rm}]_{p_r} \right) \mbox{, and}\\
\tau_{3p} (\vec {k}) &=  \left([k_1]_{3p_1},[k_2]_{3p_1}, ..., [k_m]_{3p_1}, [k_{m+1}]_{3p_2}, [k_{m+2}]_{3p_2}, ..., [k_{2m}]_{3p_2}, ...,[k_{(r-1)m+1}]_{3p_r}, ..., [k_{rm}]_{3p_r} \right).
\end{align*}

We now define $\Gamma_1$:
\begin{equation*}
\Gamma_1(f) = \ind_{\prod_{i=1}^r [-p_i, 2p_i - 1]^m} \phi \circ \tau_{3p} f \circ \tau_p.
\end{equation*}

Suppose $\vec {\theta} \in \mathbb{T}^{rm}$ and let $\vec {\xi}_i \in \mathbb{Z}^m$ so that
\begin{equation*}
\vec {\theta}= \left( \vec {\xi}_1/3p_1 + \vec {\eta}_1, \vec {\xi}_2/3p_2 + \vec {\eta}_2, ..., \vec {\xi}_r/3p_r + \vec {\eta}_r \right),
\end{equation*}
with $\left|\eta_{i,k}\right| \leq C/p_i$ for all $k$ and $i$.

If $f: \prod_{i=1}^r \mathbb{Z}_{p_i}^m \rightarrow \mathbb{R}$, then the Fourier transform of $\Gamma_1f$ would be
\begin{align}
&\widehat{\Gamma_1f}(\vec {\theta}) = \sum_{\vec {j} \in \mathbb{Z}^{rm}} \Gamma_1f(\vec {j}) e(\vec {j} \cdot \vec {\theta}) \notag \\
								&= \sum_{\vec {j} \in \prod_{i=1}^r [-p_i, 2p_i - 1]^m} \phi(\vec {j}) f \circ \tau_p(\vec {j})e(\vec {j}_1\cdot \vec {\eta}_1 + ... + \vec {j}_r \cdot \vec {\eta}_r) e\left(\vec {j}_1 \cdot \frac{\vec {\xi}_1}{3p_1} + ...+ \vec {j}_r \cdot \frac{\vec {\xi}_r}{3p_r}\right). \label{subscripthell}
\end{align}

Letting
\begin{equation*}
\psi(\vec {x}) = \phi(\vec {x}) e(\vec {x}_1\cdot \vec {\eta}_1 + ... + \vec {x}_r \cdot \vec {\eta}_r),
\end{equation*}
we find that the expression (\ref{subscripthell}) is the Fourier transform of the product of $f$ and $\psi$ on $\prod_{i=1}^r \mathbb{Z}_{3p_i}^m$.

By (\ref{MainF2}), we have that
\begin{equation*}
 \left\| \widehat{\Gamma_1 (f)}(\vec {\theta}) \right\|_{\ell^\infty} \leq \frac{1}{\prod_{i=1}^r p_{i}^m} \left\| f \right\|_{\ell^\infty} \left\| \hat{\psi} \right\|_{\ell^1}.
\end{equation*}

In this case, we have that
\begin{align*}
\left\| \hat{\psi} \right\|_{\ell^1} &= \sum_{\vec {\xi} \in  \prod_{i=1}^r \mathbb{Z}_{3p_i}^m} \left| \sum_{\vec {j} \in \prod_{i=1}^r [-p_i, 2p_i - 1]^m } \phi(\vec {j}) e(\vec {j}_1 \cdot \vec {\eta}_1 + ...+\vec {j}_r \cdot \vec {\eta}_r) e\left( \frac{\vec {j}_1 \cdot \vec {\xi}_1}{3p_1} + ...+\frac{\vec {j}_r \cdot \vec {\xi}_r}{3p_r}  \right)\right|\\
&= \sum_{\vec {\xi} \in  \prod_{i=1}^r \mathbb{Z}_{3p_i}^m} \left| \sum_{\vec {j} \in \prod_{i=1}^r [-p_i, 2p_i - 1]^m} \left(\prod_{i=1}^r \prod_{k=1}^m \varphi(j_{i,k})\right) e(\vec {j}_1 \cdot \vec {\eta}_1 + ...+\vec {j}_r \cdot \vec {\eta}_r) e\left(  \frac{\vec {j}_1 \cdot \vec {\xi}_1}{3p_1} + ...+\frac{\vec {j}_r \cdot \vec {\xi}_r}{3p_r}   \right) \right|\\
&=  \prod_{i=1}^r \prod_{k=1}^m \sum_{\xi_{i,k} \in \mathbb{Z}_{3p_i}}  \left| \sum_{j_{i,k} \in [-p_i, 2p_i - 1]} \varphi(j_{i,k})  e(j_{i,k} \eta_{i,k}) e\left( \frac{j_{i,k} \xi_{i,k}}{3p_i} \right) \right|.
\end{align*}

If
\begin{equation*}
\sum_{\xi_{i,k} \in \mathbb{Z}_{3p_i}}  \left| \sum_{j_{i,k} \in [-p_i, 2p_i - 1]} \varphi(j_{i,k})  e(j_{i,k} \eta_{i,k}) e\left( \frac{j_{i,k} \xi_{i,k}}{3p_i}  \right)\right| \leq Cp_i,
\end{equation*}
then we are done- but this is precisely the inequality (\ref{cpest}).

Suppose $m=qd$.  Define $F: \mathbb{Z}^{rm} \rightarrow \mathbb{Z}^{rd}$ by
\begin{equation*}
 F(\vec {j}) = \left( \sum_{k=1}^q p_{1}^{k-1} j_{1,k}, ...,\sum_{k=q(d-1)+1}^{dq} p_{1}^{k-(q(d-1)+1)} j_{1,k}, \sum_{k=1}^q p_{2}^{k-1} j_{2,k}, ...,\sum_{k=q(d-1)+1}^{dq} p_{r}^{k-(q(d-1)+1)} j_{r,k} \right).
\end{equation*}

Once again $F$ acts as a bijection on our sets of interest; here we have that $F$ is a bijection from $\prod_{i=1}^r [0, p_i-1]^m$ to $\prod_{i=1}^r [0, p_{i}^m-1]$.  We again define $\Gamma_2$ by
\begin{equation*}
 \Gamma_2 f(\vec {n}) = \sum_{\left\{\vec {j}: F(\vec {j}) = \vec {n} \right\}} f(\vec {j}).
\end{equation*}

As in section \ref{speckled}, we have that
\begin{equation*}
\left\| \widehat{\Gamma_2 \Gamma_1 f} \right\|_{\ell^\infty} \leq \left\| \widehat{\Gamma_1 f} \right\|_{\ell^\infty}.
\end{equation*}

Setting $\mu_r''' = \Gamma_2 \Gamma_1 \mu_r'$ and $\nu_r''' = \Gamma_2 \Gamma_1 \nu_r'$, we then have that $\mu_r'''$ and $\nu_r'''$ meet the difference requirement of Theorem \ref{differencethm}.

We also have that
\begin{equation*}
\#\left ( \supp (\Gamma_2 \Gamma_1 \mu_r')\right ) \leq \#\left ( \supp(\Gamma_1 \mu_r')\right ) \leq C\prod_{i=1}^r p_i,
\end{equation*}
and that

\begin{align*}
\left\| \nu_r''' \right\|_{\ell^1} &\leq \#\left ( \supp(\nu_r''') \right ) \left\| \Gamma_2 \Gamma_1 \nu_r' \right\|_{\ell^\infty}\\
&\leq C\prod_{i=1}^r p_{i}^m \sup_{\vec {n}} \left| \sum_{\{\vec {j}: F(\vec {j}) = \vec {n} \}} \frac{ \ind_{\prod_{i=1}^r [-p_i, 2p_i - 1]^m} \phi \circ \tau_{3p}}{\prod_{i=1}^r p_{i}^m }\right|\\
&\leq C.
\end{align*}

As in the previous section, then, all three requirements will be satisfied; noting that
\begin{equation}\label{icing}
\mu_r'''(\vec {n}) \geq \frac{1}{\prod_{i=1}^r p_i} \sum_{\vec {j} \in \tilde{S}} \delta_{\vec {j}}(\vec {n}),
\end{equation}
where
\begin{equation*}
\tilde{S} = \left\{ \left( \sum_{t=1}^{q} p_{1,k}^{t-1} [j_{1}^t]_{p_{1,k}}, ..., \sum_{t=1}^{q} p_{1,k}^{t-1} [j_{1}^{(d-1)q+t}]_{p_{1,k}},...,\sum_{t=1}^{q} p_{r,k}^{t-1} [j_{r}^{(d-1)q+t}]_{p_{r,k}}   \right) : 0 \leq j_i < p_i, 1\leq i \leq r \right\}
\end{equation*}
it remains only to reintroduce our subscripts and to shift by $\vec {a}_{i,k}$.

Defining $\mu_k$ and $\nu_k$ by
\begin{align*}
\mu_k(\vec {n}) &= \mu_{r,k}'''(\vec {n})\\
\nu_k(\vec {n}) &= \nu_{r,k}'''(\vec {n}),
\end{align*}
we have that $\sup_{k} \left| f \ast \mu_k \right|$ obeys a weak-(1,1) inequality.  With this and the inequality (\ref{icing}), the proof is complete.

\end{proof}

We have that a product of sum sets of the type constructed in Section (\ref{speckled}) remains a good sum set.

\section{Sparse Sequences and Actions of Virtually Nilpotent Groups}

We begin with a few necessary definitions.

\begin{definition}
Let $G$ be an infinite finitely generated group with identity $e$, and $\A=\{e, a_1,\dots,a_n\}\subset G$ be a finite symmetric generating set containing $e$.  Let $\A^N$ denote the elements of $G$ expressible as words of length $N$ in $\A$, and let $\A^0:=\{e\}$. Then $\rho^\A(g,h):=\min\{N:gh^{-1}\in \A^N\}$ defines a metric on $G$.
\end{definition}

Select a symmetric set of generators $\mathbb{A}$.  Then $\mathbb{A}^N$ is the ball of radius $N$ in the word metric on $G$. Classical results by Wolf \cite{Wolf}, Bass \cite{Bass}, Milnor \cite{Milnor} and Gromov \cite{MR623534} amount to the following: $G$ is virtually nilpotent (contains a nilpotent subgroup of finite index) if and only if there exists $d\in\N$ and $0<c<C<\infty$ such that for all $N\in\Z^+$,
\begin{equation}\label{Bass}
 cN^d < \# \mathbb{A}^N \leq CN^d
\end{equation}
(here $C$, but not $d$, depends on the choice of $\mathbb{A}$). Thus we say that $G$ has \emph{polynomial growth of degree $d$}.\\

Pansu \cite{Pan} improved this result further:
\begin{thm}\label{Pansu}
   Let $\mathbb{A}$ be a symmetric set of generators for the virtually nilpotent group $G$.  Then there is an integer $d$ so that the sequence
\begin{equation*}
   \frac{\# \mathbb{A}^N}{N^d}
\end{equation*}
converges.
\end{thm}

Note that in particular this implies, for all $g\in G$,
\begin{eqnarray}
\label{slippery}
\lim_{N\to\infty} \frac{\#(\A^N\Delta g\A^N)}{\#\A^N}=0.
\end{eqnarray}
Since the important matters in the proofs that follow do not depend on our choice of $\mathbb{A}$, we will henceforward suppress it in superscripts.

\subsection{Block Averages for Virtually Nilpotent Groups}\label{groupblock}

In the proof below, we will again seek to apply Tempelman's Theorem to a shifted sequence of sets whose volume increases in a lacunary fashion.  Instead of rectangular prisms, however, we will take as our shifted sets elements of the family $\{ \mathbb{A}^N \}$. \\

We immediately note that for any particular radius, the volume of the corresponding ball is finite, and that for any $m > n$, $\mathbb{A}^n \subset \mathbb{A}^m$.  Further, the family of balls of radius $N$, $\{ \mathbb{A}^N \}_{N>0}$, themselves satisfy the F\o lner condition (see \cite{Breuillard}).\\

Choose the sequence $\ell_k$ as in section \ref{speckledblock}, and a sequence of elements $a_k \in G$, with 
\begin{enumerate}
   \item $\rho(a_{k+1},e) > \rho( a_k , e) + \ell_k$, and
   \item $ \ell_k \geq C \rho( a_{k-1}, e)$.
\end{enumerate}

Defining $B_k = \mathbb{A}^{\ell_k}$, we let $S = \cup_{k>0} a_kB_k$.

Suppose 
\begin{equation*}
 S(k,r)= \left( \cup_{i < k+1} a_iB_i  \right) \cup a_{k+1}\mathbb{A}^r, 
\end{equation*}
where $0 \leq r  < \ell_k $. \\

\begin{prop}
 The sequence $S(k,r)$, with  $k \geq 1$ and $r \geq 0$, is a pointwise $L^1$-good sequence of sets for any free $G$-action.
\end{prop}

\begin{remark}
   For the proposition as it is written, merely having the upper and lower bound (as in \ref{Bass}) would suffice; however, Theorem \ref{Pansu} gives us more.  Specifically, we have that an average over any increasing sequence of sets of which our F\o lner sequence is a subsequence must also converge. For example, we have that the averages taken element-by-element also converge, so long as we successively fill each set in our constructed sequence.  A similar result holds in the $\mathbb{Z}^d$ case, as an immediate corollary.
\end{remark}

Note that if we consider only $r=0$, we would have an $L^1$-good sequence corresponding to the sequence consisting only of whole blocks in the original block sequence construction. \\

\begin{proof}
 As in Section \ref{blocks}, we need only verify the difference requirement.\\

Letting $S_k = \cup_{i < k+1} a_iB_i $ and $R = a_{k+1}\mathbb{A}^r$, we have
\begin{equation} \label{decomp1}
 \# \left( S(k,r)S^{-1}(k,r) \right)  \leq \#\left(S_kS^{-1}_k \right) + \# \left(S_kR^{-1} \right) + \# \left(R S^{-1}_k\right) + \# \left( RR^{-1}\right).
\end{equation}

We note that for any ball $\mathbb{A}^N$ in $G$, we have 
\begin{equation*}
 \mathbb{A}^N \left(\mathbb{A}^N \right)^{-1} \subseteq \mathbb{A}^{2N}.
\end{equation*}
Hence, due to the polynomial growth of $G$, the size of the difference of any ball with itself is bounded by the size of the original ball:
\begin{equation*}
 \# \left( \mathbb{A}^N \left(\mathbb{A}^N \right)^{-1}  \right) \leq \#  \mathbb{A}^{2N} \leq C\left(2N \right)^d \leq 2^d \, \# \mathbb{A}^{N}.
\end{equation*}
Thus the last term in (\ref{decomp1}) is less than $C \, \# R$ for some constant $C$. \\

For the first term, we again consider a decomposition:
\begin{align*}
 \# \left( S_k S_{k}^{-1} \right) \leq \# \left( a_kB_k B_{k}^{-1} a_{k}^{-1} \right) + \# \left(a_kB_k S_{k-1}^{-1} \right) +\# \left(S_{k-1}  B_{k}^{-1} a_{k}^{-1} \right) + \# \left( S_{k-1}S_{k-1}^{-1}\right).
\end{align*}
Again, we immediately have that the first term is less than $C \# S_k$. But $S_{k-1} \subset B_k$, by our condition on the $\ell_k$.  So each of the three other terms is also less than $C \# S_k$.\\

This leaves only the second and third terms of (\ref{decomp1}).  We note that, by our conditions on $\ell_k$,  $S_k \subseteq \mathbb{A}^{c \ell_k}$ for some constant $c$.  For the second term, then, we have
\begin{equation*}
 \# \left(S_kR^{-1} \right) \leq \# \left(\mathbb{A}^{c \ell_k} R^{-1} \right) \leq C \# \left(S(k,r)S^{-1}(k,r) \right).
\end{equation*}
In a similar way one may show that the third term is less than $C \# \left( S(k,r)S^{-1}(k,r) \right)$.

\end{proof}

\subsection{Random Averages for Measure-Preserving Group Actions}
\label{nomad}

Let $\Omega$ be a probability space, let $0<\alpha<d$, and let $\{\xi_g(\omega): g\in G\}$ be independent $\{0,1\}$-valued random variables on  $\Omega$ with $\P(\xi_g=1)=\rho(g,e)^{-\alpha}$. Note that by Theorem \ref{Pansu} and the Strong Law of Large Numbers, there exists $C$ depending on $G$, $\A$ and $\alpha$ such that $\P\left(N^{\alpha-d}\sum_{g\in\A^N}\xi_g\to C\right)=1$.  We restrict ourselves to this set $\Omega_1$ of probability 1.

\begin{definition} For a measure-preserving group action $(X,\F,m, \{T_g\})$ and$f\in L^1(X)$, define the average $$A_Nf(x):=\frac 1{\#\A^N}\sum_{g\in\A^N}f(T_gx)$$ and the random average $$A_N^{(\omega)}f(x):=N^{\alpha-d}\sum_{g\in\A^N}\xi_g(\omega)f(T_gx).$$
\end{definition}
Krengel proves several theorems about measure-preserving group actions and other additive processes in Section 6.4 of \cite{UK}.  We will apply Theorems 4.1, 4.2, and 4.4 from that section to our particular case:

\begin{thm}
\label{K}
Let $G$ have polynomial growth of degree $d$, and $\A$ be a finite symmetric generating set. Then for every measure-preserving group action $(X,\F,m,\{T_g\})$ and $1\leq p<\infty$, $A_Nf$ converges in $L^p$ and a.e. for every $f\in L^p(X,m)$.
\end{thm}
\label{K2}
Let $G$ have polynomial growth of degree $d$, and $\A$ be a finite symmetric generating set. Then we have a weak-type maximal inequality on $G$ itself,
\begin{eqnarray}
\label{costello}
\#\{g\in G: \sup_N |\varphi\ast\frac 1{\#\A^N}\1_{\A^N}|>\lambda\}\leq\frac{C}\lambda\|\varphi\|_1\hspace{10pt}\, \text {for all}\, \varphi\in\ell^1(G).
\end{eqnarray}

\noindent We may now state our main results:

\begin{thm}
\label{L2}
Let $G$ be a finitely generated group with polynomial growth of degree $d$, and $\A$ a finite symmetric generating set, and $0<\alpha<d$.  Then there exists $\Omega_2\subset\Omega$ with $\P(\Omega_2)=1$ such that for each $\omega\in\Omega_2$, $A_N^{(\omega)}f$ converges in $L^2$ and a.e. for every measure-preserving group action $(X,\F,m,\{T_g\})$ and every $f\in L^2(X,m)$.
\end{thm}
\begin{thm}
\label{L1}
Let $G$ be a finitely generated group with polynomial growth of degree $d$, and $\A$ a finite symmetric generating set, and $0<\alpha<d/2$.  Then there exists $\Omega_3\subset\Omega$ with $\P(\Omega_3)=1$ such that for each $\omega\in\Omega_3$, $A_N^{(\omega)}f$ converges in $L^1$ and a.e. for every measure-preserving group action $(X,\F,m,\{T_g\})$ and every $f\in L^1(X,m)$.
\end{thm}

\subsubsection{Proof of Theorem \ref{L2}}

\noindent The analogue of Theorem \ref{L2} was proved by Bourgain \cite{JB2} using the theory of exponential sums, and this technique extends to the natural analogues in $\Z^d$. However, on virtually nilpotent groups the Fourier transform is not so easy to work with, and so we will prove the $L^2$ theorem using the $TT^*$ method and a lemma from combinatorics.
\\
\\ It will suffice to prove convergence of the $A_N^{(\omega)}f$ along a suitable subsequence.  Indeed, fix an increasing sequence $\{N_j\}\subset\N$ such that $\frac{N_{j+1}}{N_j}\to1$. Then for any $f\geq0$ and $N_j\leq N\leq N_{j+1}$,
\begin{eqnarray}
\label{rains}
\left(\frac{N_j}{N_{j+1}}\right)^{d-\alpha}A_{N_j}^{(\omega)}f \leq A_N^{(\omega)}f \leq \left(\frac{N_{j+1}}{N_j}\right)^{d-\alpha}A_{N_{j+1}}^{(\omega)}f.
\end{eqnarray}
Then under the assumptions of Theorem \ref{L2}, it suffices to prove that $A_{N_j}^{(\omega)}f$ converges in $L^2$ and a.e. for all $f\in L^2(X)$. We may assume that $\{N_j\}$ is superpolynomial; i.e. $N_j\gg j^C$ for every $C\in\N$.
\\
\\ We will compare these random averages to their expected value, which is a weighted average of the standard ergodic averages. Define
\begin{eqnarray*}
\sigma_Nf(x):=\E_{\omega}A_N^{(\omega)}f(x)=N^{\alpha-d}\sum_{g\in\A^N}\rho(g,e)^{-\alpha}f(T_gx)=\sum_{n=0}^N a_{n,N}A_n f(x),
\end{eqnarray*}
where $a_{n,N}\geq0$, $\displaystyle\sum_{n=0}^Na_{n,N}=1$ for all $N$, and $\displaystyle \lim_{N\to\infty}a_{n,N}=0$ for all $n$.  Since $A_nf$ converges in $L^2$ and a.e. by Theorem K1, clearly $\sigma_Nf$ converges in $L^2$ and a.e. as well.
\\
\\ We will prove Theorem \ref{L2} by showing that there exists a set $\Omega_2\subset\Omega_1$ with $\P(\Omega_2)=1$ such that for every $\omega\in\Omega_2$,
\begin{eqnarray}
\label{hammer}
\|\sup_{j\geq k} |A_{N_j}^{(\omega)}f-\sigma_{N_j}f|\|_2\to0 \text{ as }k\to\infty\;\, \text {for all}\, f\in L^2(X),
\end{eqnarray}
which immediately implies $A_{N_j}^{(\omega)}f-\sigma_{N_j}f\to0$ in $L^2$ and a.e.
\\
\\ As in \cite{APET} and other papers, we hope to transfer the corresponding maximal inequality from the group algebra $\ell^p(G)$.  This Calder\'on transference principle is practically identical to the case $G=\Z$, but it is necessary to prove it in this general setting.
\begin{lem}
\label{transfer}
Let $G$ be a group with polynomial growth, and $(X,\F,m,\{T_g\})$ be a measure-preserving group action; let $\{a_{g,j}\}\subset\mathbb{C}$ such that $\sum_{g\in G} |a_{g,j}|<\infty\;\, \text {for all}\, j$. Set $A_jf=\sum_{g\in G} a_{g,j}T_gf$ and $\mu_j=\sum_{g\in G} a_{g,j}\delta_g$.
\\
\\ For any $1\leq p\leq\infty$, if $\|\sup_j |\psi\ast\mu_j|\|_p\leq C_0\|\psi\|_p\;\, \text {for all}\,\psi\in\ell^p(G)$, then $\|\sup_j |A_jf|\|_p\leq C_0\|f\|_p\;\, \text {for all}\, f\in L^p(X)$;
\\
\\ if instead $\|\sup_j |\psi\ast\mu_j|\|_{p,\infty}\leq C_0\|\psi\|_p\;\, \text {for all}\,\psi\in\ell^p(G)$, then $\|\sup_j |A_jf|\|_{p,\infty}\leq C_0\|f\|_p\;\, \text {for all}\, f\in L^p(X)$.
\end{lem}
\begin{proof}
We first consider the strong maximal inequality.  It is enough to show that $\|\sup_{1\leq j\leq J} |A_jf|\|_p\leq C_0\|f\|_p$ for all $f\in L^p(X)$, for each fixed $J\in\N$.  We may further assume that the supports of the $\mu_j$ are finite, and let $\Eset:=\bigcup_{j=1}^J \text{ supp }\mu_j$.  Take a finite symmetric set $\A$ that generates $G$, and the sets $\A^N$ defined in Section \ref{nomad}.  Fix $x\in X$ and a large finite $K\in\N$, and define $\varphi$ on $G$ by $\varphi(g)=\left\{ \begin{array}{ll} f(T_{g^{-1}}x) & \mbox{if} \; g^{-1}\in\A^K+\Eset,\\
0 & \mbox{otherwise.}  \end{array}\right.$
\\ Then $A_jf(T_gx)=\varphi\ast\mu_j(g^{-1})$ for all $g\in \A^K$ and all $j\leq J$.  This completes the proof for $p=\infty$; for $p<\infty,$
\begin{eqnarray*}
\sum_{g\in\A^K} \sup_{1\leq j\leq J} |A_jf(T_gx)|^p &=& \sum_{g\in\A^K} \sup_{1\leq j\leq J}|\varphi\ast\mu_j(g^{-1})|_p
\leq\|\sup_{k\leq j\leq J} |\varphi\ast\mu_j|\|^p_p\\
&\leq& C_0^p\|\varphi\|^p_p\\
&=&C_0^p\sum_{g\in\A^K+\Eset}|f(T_gx)|^p.
\end{eqnarray*}
Integrating over $x\in X$,
\begin{eqnarray*}
\|\sup_{1\leq j\leq J}|A_jf|\|^p_p\leq C_0^p\frac{\#(\A^K+\Eset)}{\#\A^K}\|f\|^p_p;
\end{eqnarray*}
we let $K\to\infty$ and note that (\ref{slippery}) implies (with $C_0$ independent of $J$)
\begin{eqnarray*}
\|\sup_{1\leq j\leq J}|A_jf|\|_p\leq C_0\|f\|_p.
\end{eqnarray*}
For the weak inequality, we similarly derive
\begin{eqnarray*}
\lambda^p\#\{g\in\A^K: \sup_{1\leq j\leq J}|A_jf(T_gx)|>\lambda\}\leq C_0^p\|\varphi\|^p_p
\end{eqnarray*}
and integrate this in the same manner.
\end{proof}
\noindent \textbf{Proof of Theorem \ref{L2} (Continued):} We will transfer this problem to $\ell^2(G)$ using Lemma \ref{transfer}.  Let $\eta_g(\omega)=\xi_g(\omega)-\rho(g,e)^{-\alpha}$; these are independent mean 0 Bernoulli variables.  Define for each $j$ the random measures
\begin{eqnarray}
\nu_j^{(\omega)}(g)&=&\left\{\begin{array}{ll} N_j^{\alpha-d}\eta_{g}(\omega), & g\in\A^{N_j} \\ 0, & g\not\in\A^{N_j}\end{array}\right.
\end{eqnarray}
Then for $\varphi\in \ell^p(G)$, we have the random averages $\varphi\ast\nu_j^{(\omega)}(h)=N_j^{\alpha-d}\sum_{g\in\A^{N_j}}\xi_g(\omega)\varphi(hg^{-1})$, which correspond to the operators $A_{N_j}^{(\omega)}-\sigma_{N_j}$ in the sense above. Theorem \ref{L2} therefore reduces to verifying that with probability 1 in $\Omega$, there is a sequence $C_{k,\omega}\to0$ such that

\begin{eqnarray}
\label{nail}
\|\sup_{j\geq k} |\psi\ast\nu_j^{(\omega)}|\|_2 \leq C_{k,\omega}\|\psi\|_2\; \, \text {for all}\, \psi\in\ell^2(G).
\end{eqnarray}
Since $\|\sup_{j\geq k} |\psi\ast\nu_j^{(\omega)}|\|_2^2\leq\|\sum_{j\geq k} |\psi\ast\nu_j^{(\omega)}|\|_2^2=\sum_{j\geq k} \|\psi\ast\nu_j^{(\omega)}\|_2^2$, it clearly suffices to prove that
\begin{eqnarray*}
\sum_{j=1}^\infty \|\nu_j^{(\omega)}\|^2_{op}\leq\infty,
\end{eqnarray*}
where $\|\cdot\|_{op}$ is the norm of the convolution operator on $\ell^2(G)$.
\\
\\ Since in this context we do not have the Fourier transform to help us, we will use a different Hilbert space technique: the $TT^*$ method from harmonic analysis. 
\\
\\For any operator $A$ on the Hilbert space $\ell^2(G)$, the operator norm $\|A\|=\|A^*A\|^{1/2}=\|(A^*A)^M\|^{1/2M}$; for the convolution operator $Af=\mu\ast f$, the adjoint operator is simply $A^*f=\tilde\mu\ast f$ for $\tilde\mu(g):=\overline{\mu(g^{-1})}$ ($G$ is discrete, thus unimodular).  Thus we have the trivial bound $\|A\|_{op}\leq\|(\tilde\mu\ast\mu)^M\|_{op}^{1/2M}\leq\|(\tilde\mu\ast\mu)^M\|_{\ell^1}^{1/2M}$, and thus any cancellation in the convolution products will make itself known in the original operator norm. (Here and in what follows, we use $\mu^n$ to denote the $n$-fold convolution product $\mu\ast\mu\ast\dots\ast\mu$.)
\\
\\ The cancellation in this convolution product can be described in terms of additive combinatorics on $G$: if we take a random subset $E\subset\A^N$ with size $\gg(\#\A^N)^{1/2M}$, then the number of ways to write any element of $\A^{2MN}$ as a product $g_1g_2^{-1}\dots g_{2M-1}g_{2M}^{-1}$ with all $g_i\in E$ should mostly be quite close to the ``average'' number of ways to do so. The quantitative version of this is as follows:
\begin{lem}
\label{patton}
Let $G$ be a group and $E$ a finite subset.  Let $\{X_g\}_{g\in E}$ be independent random variables with $|X_g|\leq1$ and $\E X_g=0$. Assume that $\sum_{g\in E} \Var X_g \geq1$.  Let $X$ be the random $\ell^1(G)$ function $\sum_{g\in E} X_g\delta_g$.  Then $\E \|(\tilde X\ast X)^M\|_{\ell^2}^2\leq C_M (\sum_{g\in E} \Var X_g)^{2M}$, where $C_M$ depends only on $M$.
\end{lem}
\begin{proof}
\begin{eqnarray*}
\E(\|(\tilde X\ast X)^M\|_{\ell^2}^2) & = & \E \sum_{g\in G}\left(\sum_{\scriptsize \begin{array}{c} g_1^{\,} h_1^{-1}\dots g_M^{\,}h_M^{-1}=g \\ g_i^{\,},h_i^{\,}\in E \end{array}}  X_{g_1^{\,}} X_{h_1^{\,}}\dots X_{g_M^{\,}} X_{h_M^{\,}} \right)^2\\
& = & \sum_{\scriptsize \begin{array}{c} g_1^{\null}h_1^{-1}\dots g_M^{\,}h_M^{-1}=g_{M+1}^{\null} h_{M+1}^{-1}\dots g_{2M}^{\,}h_{2M}^{-1}\\ g_i^{\,},h_i^{\,}\in E\end{array}}\E( X_{g_1^{\,}} X_{h_1^{\,}}\dots X_{g_{2M}^{\,}} X_{h_{2M}^{\,}})
\end{eqnarray*}
For any of these terms, if some $g\in  E$ appears exactly once among the $g_i$ and $h_j$, the expectation of the term will equal 0 by the independence of the $ X_g$.  Therefore we can sort the remaining terms based on the equalities between various $g_i$ and $h_{j}$; namely, in correspondence with the set partitions of $\{1,\dots,4M\}$ in which each component has size $\geq 2$.  Let there be $C_M$ of these.  For a fixed partition $\Lambda=(\lambda_1,\dots,\lambda_q)$, we can majorize the sum
\begin{eqnarray*}
\sum_{\scriptsize \begin{array}{c} (g_1^{\,},\dots,g_{2M}^{\,},h_1^{\,},\dots,h_{2M}^{\,})\in\Lambda \\ g_i^{\,},h_i^{\,}\in E \end{array}}\E(X_{g_1^{\,}}\dots X_{g_{2M}^{\,}} X_{h_1^{\,}}\dots X_{h_{2M}^{\,}})&\leq& \sum_{g_1^{\,},\dots,g_q^{\,}\in E\text{ distinct}}\E(| X_{g_1^{\,}}|^{|\lambda_1|})\dots\E(| X_{g_q^{\,}}|^{|\lambda_q|})\\
&\leq& \sum_{g_1^{\,},\dots,g_q^{\,}\in E}\E X_{g_1^{\null}}^2\dots\E X^2_{g_q^{\null}}\\
&=&(\sum_{g\in E} \Var X_g)^q\leq(\sum_{g\in E}\Var X_g)^{2M}
\end{eqnarray*}
since $\E |X_g|^p\leq\|X_g\|_\infty^{p-2}\E X_g^2\leq \E X_g^2$ for $p>2$, $\sum_{g\in E} \Var X_g\geq1$ and $q\leq 2M$.
\\ \\
Thus $\E(\|(\tilde X\ast X)^M\|_{\ell^2}^2)\leq C_M(\sum_{g\in E} \Var X_g)^{2M}.$
\end{proof}

\noindent \textbf{Proof of Theorem \ref{L2} (Conclusion):} Now by H\"older's Inequality and the fact that $\tilde\nu_j\ast\nu_j)^M$ is supported on $\A^{2MN_j}$,
\begin{eqnarray*}
\|(\tilde\nu_j\ast\nu_j)^M\|_1 &\leq& \|(\tilde\nu_j\ast\nu_j)^M\|_2(\#\A^{2MN_j})^{1/2}\leq \|(\tilde\nu_j\ast\nu_j)^M\|_2 C(2MN_j)^{d/2}.
\end{eqnarray*}
By Lemma \ref{patton}, since $\Var \eta_g\leq \rho(g,e)^{\alpha}$,
\begin{eqnarray*}
\E(\|(\tilde\nu_j^{(\omega)}\ast\nu_j^{(\omega)})^M\|_{\ell^2}^2)\leq N_j^{4M(\alpha-d)}\cdot C_M (\sum_{g\in \A^{N_j}} \Var \eta_g)^{2M}\leq C_{d,\alpha, M}N_j^{2M(\alpha-d)}
\end{eqnarray*}
and therefore by Chebyshev's Inequality,
\begin{eqnarray*}
\P(\|(\tilde\nu_j^{(\omega)}\ast\nu_j^{(\omega)})^M\|_1>\lambda)& \leq&\P\left( \|(\tilde\nu_j^{(\omega)}\ast\nu_j^{(\omega)})^M\|_2^2 C^2(2MN_j)^{d}>\lambda^2\right)\\ &\leq& C\lambda^{-2}M^dN_j^d\cdot\E(\|(\tilde\nu_j^{(\omega)}\ast\nu_j^{(\omega)})^M\|_{\ell^2}^2)\\
&\leq & C_{d,\alpha,M} \,\lambda^{-2}N_j^{2M\alpha-d(2M-1)}.
\end{eqnarray*}
As $\alpha<d$, take $M,\delta>0$ such that $d(2M-1)>2M\alpha +\delta$.  Take $\lambda=j^{-M(1+\epsilon)}$; since $N_j^\delta$ is superpolynomial, $\sum_j j^{2M(1+\epsilon)}N_j^{-\delta}<\infty$ so by the Borel-Cantelli Lemma, there is a set $\Omega_2\subset\Omega_1$ of probability 1 on which $\|(\tilde\nu_j^{(\omega)}\ast\nu_j^{(\omega)})^M\|_1<C_\omega j^{-M(1+\epsilon)}\,\, \text {for all}\, j$ and thus $\sum_{j=1}^\infty \|\nu_j^{(\omega)}\|_{op}^2\leq C_\omega \sum_{j=1}^\infty j^{-1-\epsilon}<\infty$.  This completes the proof of Theorem \ref{L2}.

\subsubsection{Proof of Theorem \ref{L1}}
By Theorem \ref{L2}, for $\omega\in\Omega_2$ we have a.e. convergence of $A_N^{(\omega)}f$ for $f\in L^2(X)$, which is dense in $L^1(X)$.  We therefore need only a weak type maximal inequality to prove Theorem \ref{L1}.  As usual, it is enough to consider the dyadic subsequence $2^j$.  Now for $f\geq0$, $0\leq A^{(\omega)}_{N}f \lesssim A^{(\omega)}_{2^{j+1}}f$ for $2^j\leq N<2^{j+1}$, so it suffices to prove
\begin{eqnarray}
\|\sup_j |A^{(\omega)}_{2^j}f|\|_{1,\infty}\leq C\|f\|_1 \;\, \text {for all}\, f\in L^1(X).
\end{eqnarray}
Again, we will use Lemma \ref{transfer} to transfer this maximal inequality from $\ell^1(G)$.  Let
\begin{eqnarray*}
\mu_j^{(\omega)}(g)&:=&\left\{\begin{array}{ll} 2^{(\alpha-d)j}\xi_{g}(\omega), & g\in\A^{2^j} \\ 0, & g\not\in\A^{2^j}\end{array}\right.\\
\E\mu_j(g)&:=&\left\{\begin{array}{ll} 2^{(\alpha-d)j}\E\xi_{g}, & g\in\A^{2^j} \\ 0, & g\not\in\A^{2^j}\end{array}\right.\\
\nu_j^{(\omega)}(g)&:=&\mu_j^{(\omega)}(g)-\E\mu_j^{(\omega)}(g);
\end{eqnarray*}
$\mu_j^{(\omega)}$ and $\E\mu_j$ correspond to the operators $A_{2^j}^{(\omega)}$ and $\sigma_{2^j}$, respectively.  Theorem \ref{L1} reduces to proving
\begin{eqnarray}
\| \sup_j | \varphi\ast\mu_{j}^{(\omega)} |\|_{1,\infty}\leq C_\omega\| \varphi\|_1.
\end{eqnarray}

\begin{prop}
\label{shine on}
Let $\mu_j$ and $\nu_j$ be sequences of functions in $\ell^1(G)$, where $G$ has polynomial growth of degree $d$.  Let $r_j:=\#\{g:\mu_j(g)\neq0\}$ and take $R_j:=\inf\{R>0: \nu_j(g)\neq0\implies \rho(g,e)\leq R\}$.  Assume there exists $C_0<\infty$ such that $\sum_{j\leq k} r_j\leq C_0r_k\;\, \text {for all}\, k\in\N$, and that
\begin{eqnarray}
\label{convo}
\nu_j\ast\tilde\nu_j=O(r_j^{-1})\delta_e+O(R_j^{-d-\epsilon})\text{ for some }\epsilon>0.
\end{eqnarray}
If $\, \text {for all}\, \varphi$, $\| \displaystyle\sup_j \varphi\ast |\mu_j-\nu_j|\|_{1,\infty}\leq C\|\varphi\|_1$ and $\| \displaystyle\sup_j|\varphi\ast\mu_j|\|_{p,\infty}\leq C_p\|\varphi\|_p$ for some $1<p\leq\infty,$ then
\begin{eqnarray}
\label{sup}
\|\sup_j|\varphi\ast\mu_j|\|_{1,\infty}\leq C'\|\varphi\|_1 \;\, \text {for all}\,\varphi\in\ell^1(G).
\end{eqnarray}
\end{prop}
\begin{proof}
This is simply an extension of the proof of Theorem \ref{SgtPepper}; however, we must first establish that the Calder\'on-Zygmund decomposition makes sense on more general groups $G$. Since word-length is a quasimetric on $G$, we can use the $\rho$-dyadic cubes constructed by Christ in \cite{MC2} on spaces of homogeneous type.  Namely, there exist a collection of subsets $\{Q_{s,k}\subset G: s\in \N, k\in\Z\}$, and constants $A>1, a_0>0, C_1<\infty$ such that
\begin{eqnarray}
&\forall s\in\N,\; G=\bigcup_k Q_{s,k} \\
&r\leq s \implies Q_{r,l}\subset Q_{s,k} \text{ or } Q_{r,l}\cap Q_{s,k}= \emptyset\\
&\, \forall\,(r,l), \, \forall s>r,\; \exists! k\in\Z \text{ such that } Q_{r,l}\subset Q_{s,k}\\
\label{diameter}
&\text{Diameter } Q_{s,k}\leq C_1A^s\\
&\text{Each }Q_{s,k}\text{ contains some ball of radius }a_0 A^s.
\end{eqnarray}

Because $G$ has a polynomial rate of growth, $\rho$ is a doubling metric, and thus we can prove the Vitali Covering Lemma and the Hardy-Littlewood Maximal Inequality on $G$.  Using a standard stopping-time argument, we can then define a suitable discrete Calder\'on-Zygmund decomposition on $G$ with the dyadic cubes.
\\
\\ Fix $\lambda>0$.  We take $\varphi ={\mathfrak g}+b$, where $\|{\mathfrak g}\|_\infty\leq\lambda$ and $b=\displaystyle\sum_{(s,k)\in\B} b_{s,k}$ for some index set $\B\subset \N^2$, where $b_{s,k}$ is supported on $Q_{s,k}$, $\{Q_{s,k}:(s,k)\in\B\}$ is a disjoint collection, $\|b_{s,k}\|_1\leq\lambda\#Q_{s,k}$ and $\displaystyle \sum_{(s,k)\in\B}\#Q_{s,k}\leq \frac{C}\lambda\| \varphi\|_1$ ($C$ independent of $\varphi$ and $\lambda$).  Let $b_s=\displaystyle\sum_k b_{s,k}$.
\\
\\We further decompose $b_{s,k}=b^{(j)}_{s,k}+B^{(j)}_{s,k}$, where $b^{(j)}_{s,k}=b_{s,k}\1(|b_{s,k}|>\lambda r_j)$.  Define $b^{(j)}_s, B^{(j)}_s, b^{(j)}, B^{(j)}$ by summing over one or both indices, respectively.
\\
\\ We will divide $B^{(j)}=\sum_s B^{(j)}_s$ into two parts, splitting at the index $s(j):=\min\{s: A^s\geq R_j\}$.
\\
\\Now $\{g: \sup_j | \varphi\ast\mu_j (g)|> 5\lambda\}\subset E_1\cup E_2\cup E_3\cup E_4\cup E_5$, where
\begin{eqnarray*}
E_1&=&\{g: \sup_j |{\mathfrak g} \ast\mu_{j} (g)|> \lambda\}\\
E_2&=&\{g: \sup_j |b^{(j)}\ast\mu_j  (g)|> \lambda\}\\
E_3&=&\{g: \sup_j |B^{(j)}\ast(\mu_{j}-\nu_j) (g)|> \lambda\}\\
E_4&=&\{g: \sup_j |\left(\sum_{s=s(j)}^\infty B_{s}^{(j)} \right)\ast\nu_j(g)|> \lambda\}\\
E_5&=&\{g: \sup_j | \left(\sum_{s=0}^{s(j)-1}B_{s}^{(j)} \right)\ast\nu_j(g)|> \lambda\}
\end{eqnarray*}
By the weak $(p,p)$ inequality (if $p<\infty$), $\#E_1\leq C\lambda^{-p}\|{\mathfrak g}\|^p_p\leq C\lambda^{-p}\|{\mathfrak g}\|_\infty^{p-1}\|{\mathfrak g}\|_1\leq C\lambda^{-1}\|\varphi\|_1$; if $p=\infty$, re-do the decomposition so that $\|{\mathfrak g}\|_\infty< C_\infty^{-1}\lambda$ instead; then $E_1$ will be empty since $\|\sup_j |{\mathfrak g} \ast\mu_{j}|\|_\infty\leq C_\infty\|{\mathfrak g}\|_\infty\leq\lambda$.
\\
\\Next,
\begin{eqnarray*}
\#E_2\leq\sum_j \#\{g:|b^{(j)}\ast\mu_{j} (g)|>0\}&\leq&\sum_j \#(\text{supp }\mu_{j})\cdot\#\{g: |b(g)|>\lambda r_j\}\\
&=&\sum_j r_j \sum_{k\geq j} \#\{g:\lambda r_k<|b(g)|\leq \lambda r_{k+1}\}\\
&=&\sum_{k} \#\{g:\lambda r_k<|b(g)|\leq \lambda r_{k+1}\} \sum_{j\leq k}r_j\\
&\leq&\frac{C_0}\lambda \sum_k \lambda r_k \#\{g:\lambda r_k<|b(g)|\leq \lambda r_{k+1}\};
\end{eqnarray*}
 now note that this sum is a lower sum for $|b|$, and we have $\#E_2\leq C_0\lambda^{-1}\|b\|_1\leq\frac{C}\lambda\|\varphi\|_1$.
\\
\\For $E_3$, $|B^{(j)}\ast(\mu_j-\nu_j) (g)|\leq |B^{(j)} |\ast|\mu_j-\nu_j|(g)\leq|b|\ast|\mu_j-\nu_j|(g)$, so by the weak $(1,1)$ inequality,
\begin{eqnarray*}
\#E_3\leq \#\{g:\sup_j |b|\ast|\mu_j-\nu_j|(g)>\lambda\}\leq \frac{C}\lambda \|b\|_1\leq \frac{C}\lambda\|\varphi\|_1.
\end{eqnarray*}
To bound $E_4$, note that for all $s\geq s(j)$, $B_{s,k}^{(j)}\ast\nu_j $ is supported on  $Q^*_{s,k}:=\{g: \rho(g, Q_{s,k})\leq A^s\}$, so
\begin{eqnarray*}
\# E_4\leq\sum_{(s,k)\in\B} C\#Q_{s,k}\leq\frac{C}\lambda\|\varphi\|_1.
\end{eqnarray*}
We have thus reduced the problem to obtaining a bound on the size of $E_5$.

\begin{lem}
Let $B_s^{(j)}$ be as above, and assume the $\nu_j$ satisfy (\ref{convo}).  For $0\leq s<s(j)$,
\begin{eqnarray*}
\|B_{s}^{(j)}\ast\nu_j \|_{\ell^2(G)}^2\leq C r_j^{-1}\|B_{s}^{(j)}\|_2^2 + C\lambda 2^{-\epsilon j}\|B_{s}^{(j)}\|_1
\end{eqnarray*}
and for $0\leq s_1<s_2<s(j)$,
\begin{eqnarray*}
|\langle B_{s_1}^{(j)}\ast\nu_j ,B_{s_2}^{(j)}\ast\nu_j \rangle_{\ell^2(G)}|\leq C \lambda 2^{-\epsilon j} \|B_{s_2}^{(j)}\|_1.
\end{eqnarray*}
\end{lem}
\begin{proof}
We first restrict the supports of the $B_s$; we assume there is a $Q_{s(j),k_0}$ such that $Q_{s,k}\subset Q_{s(j),k_0}$ for all $(s,k)\in\B$ with $s<s(j)$.  Then $\|B_s^{(j)}\|_1\leq \|b_s\|_1\leq\sum_{(s,k)\in\B}\lambda|Q_{s,k}|\leq\lambda |Q_{s(j),k_0}|\leq C\lambda R_j^{d}$, and thus
\begin{eqnarray*}
|\langle B_{s_1}^{(j)}\ast\nu_j ,B_{s_2}^{(j)}\ast\nu_j \rangle|&=&|\langle B_{s_1}^{(j)}\ast\nu_j\ast \tilde\nu_j, B_{s_2}^{(j)}\rangle|\\
&\leq& Cr_j^{-1}|\langle B_{s_1}^{(j)}, B_{s_2}^{(j)}\rangle|+ CR_j^{-d}2^{-\epsilon j}\|B_{s_1}^{(j)}\|_1 \|B_{s_2}^{(j)}\|_1\\
&\leq& Cr_j^{-1}|\langle B_{s_1}^{(j)}, B_{s_2}^{(j)}\rangle|+ C\lambda 2^{-\epsilon j}\|B_{s_2}^{(j)}\|_1.
\end{eqnarray*}
Now this first term is 0 if $s_1\neq s_2$, and $Cr_j^{-1}\|B_{s_1}^{(j)}\|_2^2$ if $s_1=s_2$.
\\
\\We remove the assumption on the supports by noting that if the distance between the supports of $\varphi_1$ and $\varphi_2$ is greater than $2R_j$, then $\langle \varphi_1\ast\nu_j, \varphi_2\ast\tilde\nu_j\rangle=0$.  Thus if we decompose each $B_s=\sum_k B_s\1(Q_{s(j),k})$ and decompose the inner products accordingly, all but finitely many of the terms (a number independent of $j$) will vanish; and those remaining can be estimated in this way.
\end{proof}
\noindent Now by Chebyshev's Inequality,
\begin{eqnarray}
\lambda^2\#\{g: \sup_j |\sum_{s=0}^{s(j)-1} B_{s}^{(j)}\ast\nu_j (g)|>\lambda\}&\leq&\sum_g\sup_j |\sum_{s=0}^{s(j)-1} B_{s}^{(j)}\ast\nu_j (g)|^2 \notag \\
\label{money}
&\leq&\sum_j \|\sum_{s=0}^{s(j)-1} B_{s}^{(j)}\ast\nu_j\|_2^2\\
&\leq& \sum_j\sum_{\substack{ s_1,s_2:\\ 0\leq s_1,s_2 < s(j) }} |\langle B_{s_1}^{(j)}\ast\nu_j ,B_{s_2}^{(j)}\ast\nu_j \rangle_{\ell^2(G)}| \notag
\end{eqnarray}
and this is
\begin{eqnarray*}
&&\leq \sum_j\sum_{s=0}^{s(j)-1} \left(C r_j^{-1}\|B_{s}^{(j)}\|_2^2 + C\lambda 2^{-\epsilon j}\|B_{s}^{(j)}\|_1\right)+2\sum_j\sum_{\substack{s_1,s_2:\\ 0\leq s_1<s_2 < s(j)}} C \lambda 2^{-\epsilon j} \|B_{s_2}^{(j)}\|_1\\
&&\leq \sum_{s=0}^\infty \sum_{j=1}^\infty C\lambda (1+j)2^{-\epsilon j}\|B_s^{(j)}\|_1+ \sum_j\sum_{s=0}^{s(j)-1} Cr_j^{-1}\|B_{s}^{(j)}\|_2^2\\
&&\leq \sum_{s=0}^\infty C\lambda \|b_s\|_1+\sum_j\sum_{s=0}^{s(j)-1} Cr_j^{-1}\|B_{s}^{(j)}\|_2^2.
\end{eqnarray*}
The first term is $\leq C\lambda\|\varphi\|_1$ as desired.  For the second term, note that
$$\sum_{j\leq k} r_j\leq C_0r_k\;\, \text {for all}\, k\in\N\implies \exists N \text{ s.t. } r_{j+n}\geq 2r_j \, \text {for all}\, j\in\N,n\geq N\implies \sum_{j=k}^\infty r_j^{-1}\leq Cr_k^{-1}.$$
Since the $Q_{s,k}$ are disjoint, for a fixed $g\in Q_{s_0,k_0}$,
\begin{eqnarray*}
\sum_j\sum_{s=0}^{s(j)-1}  r_j^{-1}|B_{s}^{(j)}(g)|^2\leq\sum_{\scriptsize \begin{array}{c} j:\\ \lambda r_j\geq |b_{s_0}(g)| \end{array}}r_j^{-1}|b_{s_0}(g)|^2\leq C\lambda|b_{s_0}(g)|=C\lambda|b(g)|
\end{eqnarray*}
so $\sum_j\sum_{s=0}^{j-1} Cr^{-j}\|B_{s}^{(j)}\|_2^2\leq C\lambda\|b\|_1\leq C\lambda\|\varphi\|_1$ and the proof of (\ref{sup}) is complete.
\end{proof}
\noindent Having established Proposition \ref{shine on}, it remains to show that the random measures $\mu_j^{(\omega)}$ and $\nu_j^{(\omega)}$ satisfy the assumptions with probability 1.  Note first that  $r_j=|\text{supp }\mu_j^{(\omega)}|=\sum_{g\in\A^{2^j}}\xi_g(\omega)\lesssim 2^{(d-\alpha)j}$ on $\Omega_1$, and $\nu_j^{(\omega)}$ is supported on $\A^{2^j}$ with $\rho$-diameter at most $R_j=2^{j+1}$. We must prove the bound (\ref{convo}) on $\nu_j^{(\omega)}\ast \tilde\nu_j^{(\omega)}$.

\begin{lem}
\label{cancels}
Let $G$ be a group and $E$ a finite subset.  Let $\{  X_g\}_{g\in E}$ be independent random variables with $|X_g|\leq1$ and $\E X_g=0$.  Assume that $\sum_{g\in E} (\Var X_g)^2\geq1$.  Let $X$ be the random $\ell^1(G)$ function $\sum_{g\in E} X_g\delta_g$.  Let $G^\times$ denote $G\setminus\{e\}$.  Then for any $\theta>0$,
\begin{eqnarray}
\P\left(\| X\ast\tilde X\|_{\ell^\infty(G^\times)}\geq \theta(\sum_{g\in E} (\Var X_g)^2)^{1/2}\right)\leq 6|{E}|^2\max(e^{-\theta^2/36},e^{-\theta/6}).
\end{eqnarray}
\end{lem}
\begin{proof}
For $h\neq e$,
$$X\ast\tilde X(h)=\sum_{g\in E\cap h^{-1} E}X_gX_{gh}=\sum_{g\in E\cap h^{-1} E} Y_g$$
where $\E Y_g=0$ and $|Y_g|\leq1$.  We want to apply Chernoff's Inequality, but the $Y_g$ are not independent.
\\
\\We can, however, partition $E\cap h^{-1}E$ into at most three subsets $E_1,E_2,E_3$, in each of which the $Y_g$ are independent.  To see this, note that we can make a directed graph with vertex set $E$ and edge set $\{(g,hg):g, hg\in E\}$; and that the components of this graph are paths or cycles.  Thus we can three-color this graph; and within each resulting $E_i$, the $Y_g$ depend on distinct independent random variables, so they are independent.
\\
\\ Now $\displaystyle \sum_{g\in E_i} Y_g$ has variance 
\begin{eqnarray*}
\displaystyle \sigma^2=\sum_{g\in E_i}\Var X_g \Var X_{gh}\leq\sum_{g\in E_i}(\Var X_g)^2\leq \sum_{g\in E}(\Var X_g)^2
\end{eqnarray*}
by H\"older's Inequality.  Chernoff's Inequality (Theorem 1.8 in \cite{TV}) gives us 
\begin{eqnarray*}
\P(|\sum_{g\in E_i} Y_g|\geq\lambda\sigma)\leq2\max(e^{-\lambda^2/4},e^{-\lambda\sigma/2}).
\end{eqnarray*}
Take $\lambda=\theta\sigma^{-1}(\sum_{g\in E} (\Var X_g)^2)^{1/2}$; then $\lambda\geq\theta$ and $\lambda\sigma=\theta(\sum_{g\in E} (\Var X_g)^2)^{1/2}\geq\theta$, so
$$\P(|X\ast\tilde X(h)|\geq 3\theta(\sum_{g\in E} (\Var X_g)^2)^{1/2})\leq\sum_{i=1}^3\P(|\sum_{E_i} Y_g|\geq\lambda\sigma)\leq6\max(e^{-\theta^2/4},e^{-\theta/2}).$$
Since this holds for each $h\neq e$ and $|\text{supp }X\ast\tilde X|\leq |E|^2$, the conclusion follows (after replacing $3\theta$ with $\theta$).
\end{proof}

\begin{cor}
\label{bound}
Let $\nu_j^{(\omega)}$ be the random measure defined as before, $0<\alpha<d/2$ and $\kappa>0$.  Then there is a set $\Omega_3\subset\Omega_2$ with $\P(\Omega_3=1)$ such that for each $\omega\in\Omega_3$,
\begin{eqnarray}
\nu_j^{(\omega)}\ast \tilde\nu_j^{(\omega)}= O_\omega(2^{(\alpha-d)j})\delta_e+O_\omega(2^{2(\alpha-d)j}(\sum_{g\in\A^{2^j}}\E\xi_g^2)^{1/2}2^{\kappa j}).
\end{eqnarray}
\end{cor}
\begin{proof}
For the bound at the identity $e$, we use the fact that
\begin{eqnarray*} 
\nu_j^{(\omega)}\ast\tilde\nu_j^{(\omega)}(e)&=&2^{2(\alpha-d)j}\sum_{g\in\A^{2^j}}\eta_g^2(\omega)\\
&\leq&2^{2(\alpha-d)j}\sum_{g\in\A^{2^j}}(\E\xi_g+\xi_g(\omega))=2^{(\alpha-d)j+1}+2^{2(\alpha-d)j}\sum_{g\in\A^{2^j}}\eta_g(\omega)
\end{eqnarray*}
so that
\begin{eqnarray*}
\P(\nu_j^{(\omega)}\ast\tilde\nu_j^{(\omega)}(e)>3\cdot 2^{(\alpha-d)j})\leq\P(\sum_{g\in\A^{2^j}} \eta_{g}(\omega) >2^{(d-\alpha)j})\leq 2\exp(-\frac12 2^{(d-\alpha)j})
\end{eqnarray*}
 for $j$ sufficiently large, by Chernoff's inequality.  The Borel-Cantelli Lemma then implies that $\nu_j^{(\omega)}\ast\tilde\nu_j^{(\omega)}(e)\leq 3\cdot 2^{(\alpha-d)j}$ for $j$ sufficiently large (depending on $\omega$), so there exists $C_\omega$ with $0\leq\nu_j^{(\omega)}\ast\tilde\nu_j^{(\omega)}(e)\leq C_\omega 2^{(\alpha-d)j}$ for all $j$.
\\
\\ For the other term, we note that $\Var \eta_g\leq \E\xi_g$, so we set $\theta=2^{\kappa j}$ and apply Lemma \ref{cancels}:
\begin{eqnarray*}
\P\left(2^{2(d-\alpha)j}\|\nu_j^{(\omega)}\ast\tilde\nu_j^{(\omega)}\|_{\ell^\infty(G^\times)}\geq 2^{\kappa j}(\sum_{g\in\A^{2^j}}\E\xi_g^2)^{1/2}\right)\leq C2^{2dj}\exp(-2^{\kappa j}/2)
\end{eqnarray*}
which sum over $j$.  The Borel-Cantelli Lemma again proves the bound holds with probability 1.
\end{proof}
\noindent Note that $\displaystyle\sum_{g\in\A^{2^j}}\E\xi_g^2\lesssim 2^{(d-2\alpha)j};$ thus for $\alpha<d/2$, 
\begin{eqnarray*}
2^{2(\alpha-d)j}\left(\sum_{g\in\A^{2^j}}\E\xi_g^2\right)^{1/2}2^{\kappa j}\leq C2^{(-\frac{3d}2+\alpha+\kappa)j}=CR_j^{-d}2^{(-\frac{d}2+\alpha+\kappa)j}
\end{eqnarray*} 
and thus for $\kappa$ chosen small, the measures $\nu_j^{(\omega)}$ satisfy the bound (\ref{convo}) for all $\omega\in\Omega_3$.  Since $\mu_j^{(\omega)}-\nu_j^{(\omega)}=\E\mu_j$ is a weighted average of the nonnegative averages in (\ref{costello}), Theorem K2 implies $\|\sup_j |\varphi\ast\E\mu_j|\|_{1,\infty}\leq C\|\varphi\|_1$; and the $\ell^\infty$ maximal inequality for $\mu_j^{(\omega)}$ is trivial.  Thus Proposition \ref{shine on} applies, and we have proved Theorem \ref{L1}.
\begin{remark}
As in the case of $\Z^d$, this method is inherently limited to exponents $\alpha<d/2$, because otherwise the set is too sparse for a single convolution product to be ``uniformly'' small in any nontrivial sense.
\end{remark}

\subsection{Gaps and Banach Density}\label{gaps}

In \cite{LaVic1} it was noted that, with probability 1, the sparse random sequences in $\N$ have Banach density 0, which distinguishes them from block sequences of the Bellow-Losert type in Section \ref{blocks}. This remains true for the random subset $\{g\in G: \xi_g(\omega)=1\}$ which we have obtained.
\\
\\It is worth noting a second distinction: with probability 1, this random set has a subset with gaps tending to infinity which is of full relative measure, and thus the averages over this subset still converge a.e. for functions in $L^1$. (Note that this is not a necessary consequence of Banach density 0: consider the ``Cantor set'' of natural numbers that can be written as finite sums of distinct powers of 3. This set has Banach density 0, but no set of positive relative measure can have gaps tending to $\infty$.)\\
\\
We will order the random set $\{g_n\}=\{g\in G: \xi_g(\omega)=1\}$ so that $\rho(g_n,e)$ is nondecreasing. The convergence of the ergodic averages $A_{N_j}^{(\omega)}$ in (\ref{rains}) implies that we can add the points one at a time and maintain the pointwise ergodic theorem; that is, the averages
\begin{eqnarray*}
A_N^{(g_n)}f(x):=\frac1N\sum_{n=1}^N f({\cal T}(g_n)x)
\end{eqnarray*}
converge a.e. in $X$ for any measure-preserving group action $(X,{\cal T})$.
\begin{definition}
For $j\geq0$ and $M<\infty$, let
\begin{eqnarray}
\Gamma_{j,M}&:=&\{n\in[2^j,2^{j+1}):\inf_{m<n} \rho(g_n,g_m)<M\}\\
\beta_{t,M}&:=&2^{-j} |\Gamma_{t,M}|.
\end{eqnarray}
\end{definition}
\begin{prop}
Let $\{g_n\}$ be a sequence in a virtually nilpotent discrete group $G$, such that $\rho(g_n,e)$ is nondecreasing and the averages $A_N^{(g_n)}f$ converge a.e. for all $f\in L^1(X)$. If $\displaystyle \sum_{j} \beta_{j,M}<\infty$ for every $M<\infty$, then there exists an increasing sequence $\{n_k\}\in\N$ such that $\{g_{n_k}\}$ has gaps tending to infinity, and such that the averages $A_N^{(g_{n_k})}f$ converge a.e. for all $f\in L^1(X)$.
\end{prop}
\begin{proof}
We can clearly choose a sequence $M_j$ with $M_j\to\infty$ so that $\sum_j \beta_{j,M_j}<\infty$. Let $\{n_k\}$ be the set $\N\setminus(\bigcup_j \Gamma_{j,M_j})$, taken in increasing order, and note that $\frac{n_k}{k}\to1$.\\
\\
Considering $[1,n_K]$ as the union of terms in $\{n_k\}$ and the complement, we see
\begin{eqnarray*}
\left|A_K^{(g_{n_k})}f(x)-\frac{n_K}{K} A_{n_K}^{(g_n)}f(x)\right|\leq \left|\frac1{K}\sum_{\substack{ n<n_K: \\n\in \bigcup_j \Gamma_{j,M_j}}} f({\cal T}(g_n)x)\right|. 
\end{eqnarray*}
Clearly the $L^1$ norm of the right-hand side is bounded by $\frac{n_K-K}{K}\|f\|_1\to0$ as $K\to\infty$. Furthermore, we have the weak maximal inequality
\begin{eqnarray*}
\left\|\sup_K\left|\frac1{K}\sum_{\substack{ n<n_K: \\n\in \bigcup_j \Gamma_{j,M_j}}} f({\cal T}(g_n)x)\right|\right\|_1&\leq& \left\|\sup_j \left|2^{-j}\sum_{\substack{ n<2^{j}: \\n\in \bigcup_j \Gamma_{j,M_j}}} f({\cal T}(g_n)x)\right| \right\|_1\\
&\lesssim& \sum_j \beta_{j,M_j}<\infty.
\end{eqnarray*} 
Therefore $A_K^{(g_{n_k})}f-\frac{n_K}{K} A_{n_K}^{(g_n)}f\to0$ a.e. and since $\frac{n_K}{K}\to1$, this proves that $A_N^{(g_{n_k})}f$ converge a.e. for all $f\in L^1(X)$.
\end{proof}
It remains to show that our randomly generated sequences indeed have this property.
If $B$ is the ball of radius $2M$ centered at $g$, and $\rho(g,e)\approx 2^l\gg M$, then clearly
\begin{eqnarray}
\label{easy}
\P\left( \sum_{g\in B}\xi_g = t\right)\leq {|B|\choose t}2^{-\gamma tl}.
\end{eqnarray}
If $\B_l$ is a cover of $\left\{ g\in G:\rho(g,e)\approx 2^l \right\}$ by balls of radius $2M$, with multiplicity of intersection controlled uniformly in $l$, then for $j=(d-\gamma)l$,
\begin{eqnarray*}
\beta_{j,M}&\lesssim& 2^{-j}|{\B}_l|\sum_{t=2}^{|B|} (t-1){|B|\choose t}2^{-\gamma tl}\\
&\leq& C_M|{\B}_l|2^{-2\gamma l-j}\\
&\leq& C_M 2^{l(d-2\gamma)-j}\leq C_M 2^{-\gamma l},
\end{eqnarray*}
and these are summable.

\begin{remark}
Let $\{\vec n_i\}$ be one of the sparse deterministic sequences from Section \ref{detsparse} (either speckled or plaid); because of the nature of the Freiman isomorphism $F_p$ in (\ref{Freiman}) and the pseudorandomness of the points $(j,j^2,\dots,j^m)\in\Z_p^m$, it becomes vanishingly rare for two points in the $k$th block to be within $p_k$ of each other, and thus the $\beta_{j,M}$ are summable as well. Therefore these sequences can be modified in a negligible manner so as to have gaps tending to infinity.
\end{remark}

\bibliography{MultivarSparseAvg}
\bibliographystyle{plain}

\bigskip

{\small
\parbox[t]{3in}
{P. LaVictoire \\Department of Mathematics \\
University of Wisconsin\\
Madison, WI 53706\\
E-mail: patlavic@math.wisc.edu\\}
\bigskip

{\small
\parbox[t]{3in}
{A. Parrish \\Department of Mathematics \\
University of Illinois at Urbana-Champaign \\
Urbana, IL 61801\\
E-mail: ajnparrish@gmail.com\\}
\bigskip

{\small
\parbox[t]{3in}
{J. Rosenblatt\\Department of Mathematics \\
University of Illinois at Urbana-Champaign \\
Urbana, IL 61801\\
E-mail: rosnbltt@illinois.edu\\}
\bigskip

\end{document}